\documentclass[11pt,a4paper]{article}
\usepackage{graphicx}
\usepackage{mathrsfs}
\usepackage{color}
\usepackage[latin1]{inputenc}
\usepackage[english]{babel}

\usepackage{libertine}
\usepackage[T1]{fontenc}
\usepackage{lipsum}
\usepackage{titlesec}
\usepackage[colorlinks=true,linkcolor=magenta,citecolor=blue]{hyperref}

\usepackage[cmex10]{amsmath}
\usepackage{stmaryrd}
\usepackage{color}
\usepackage{cmll}
\usepackage{amsmath,amssymb}
\usepackage[sloped]{fourier}
\usepackage{amsthm}
\usepackage{esint}
\usepackage{psfrag}

\author{S{\'e}bastien Alvarez\thanks{S.A. was supported by a post-doctoral grant financed by CAPES} \and Jiagang Yang \thanks{J.Y was partially supported by CNPq, FAPERJ, and PRONEX}}

\date{}
\title{Physical measures for the geodesic flow tangent to a transversally conformal foliation}

\begin{document}

\newtheorem{maintheorem}{Theorem}
\newtheorem{maincoro}[maintheorem]{Corollary}
\renewcommand{\themaintheorem}{\Alph{maintheorem}}
\newcounter{theorem}[section]
\newtheorem{exemple}{\bf Exemple \rm}
\newtheorem{exercice}{\bf Exercice \rm}
\newtheorem{conj}[theorem]{\bf Conjecture}
\newtheorem{defi}[theorem]{\bf Definition}
\newtheorem{lemma}[theorem]{\bf Lemma}
\newtheorem{proposition}[theorem]{\bf Proposition}
\newtheorem{coro}[theorem]{\bf Corollary}
\newtheorem{theorem}[theorem]{\bf Theorem}
\newtheorem{rem}[theorem]{\bf Remark}
\newtheorem{ques}[theorem]{\bf Question}
\newtheorem{propr}[theorem]{\bf Property}
\newtheorem{question}{\bf Question}
\def\bp{\noindent{\it Proof. }}
\def\ep{\noindent{\hfill $\fbox{\,}$}\medskip\newline}
\renewcommand{\theequation}{\arabic{section}.\arabic{equation}}
\renewcommand{\thetheorem}{\arabic{section}.\arabic{theorem}}
\newcommand{\eps}{\varepsilon}
\newcommand{\disp}[1]{\displaystyle{\mathstrut#1}}
\newcommand{\fra}[2]{\displaystyle\frac{\mathstrut#1}{\mathstrut#2}}
\newcommand{\dif}{{\rm Diff}}
\newcommand{\homeo}{{\rm Homeo}}
\newcommand{\Per}{{\rm Per}}
\newcommand{\Fix}{{\rm Fix}}
\newcommand{\A}{\mathcal A}
\newcommand{\Z}{\mathbb Z}
\newcommand{\Q}{\mathbb Q}
\newcommand{\R}{\mathbb R}
\newcommand{\C}{\mathbb C}
\newcommand{\N}{\mathbb N}
\newcommand{\T}{\mathbb T}
\newcommand{\U}{\mathbb U}
\newcommand{\D}{\mathbb D}
\newcommand{\PP}{\mathbb P}
\newcommand{\Sp}{\mathbb S}
\newcommand{\K}{\mathbb K}
\newcommand{\car}{\mathbf 1}
\newcommand{\g}{\mathfrak g}
\newcommand{\gs}{\mathfrak s}
\newcommand{\h}{\mathfrak h}
\newcommand{\rr}{\mathfrak r}
\newcommand{\fhi}{\varphi}
\newcommand{\ffhi}{\tilde{\varphi}}
\newcommand{\moins}{\setminus}
\newcommand{\ds}{\subset}
\newcommand{\W}{\mathcal W}
\newcommand{\WW}{\widetilde{W}}
\newcommand{\F}{\mathcal F}
\newcommand{\G}{\mathcal G}
\newcommand{\CC}{\mathcal C}
\newcommand{\RR}{\mathcal R}
\newcommand{\DD}{\mathcal D}
\newcommand{\M}{\mathcal M}
\newcommand{\B}{\mathcal B}
\newcommand{\cS}{\mathcal S}
\newcommand{\HH}{\mathcal H}
\newcommand{\NN}{\mathcal N}
\newcommand{\Hyp}{\mathbb H}
\newcommand{\UU}{\mathcal U}
\newcommand{\Pp}{\mathcal P}
\newcommand{\QQ}{\mathcal Q}
\newcommand{\E}{\mathcal E}
\newcommand{\GG}{\Gamma}
\newcommand{\LL}{\mathcal L}
\newcommand{\OO}{\mathcal O}
\newcommand{\KK}{\mathcal K}
\newcommand{\TT}{\mathcal T}
\newcommand{\X}{\mathcal X}
\newcommand{\Y}{\mathcal Y}
\newcommand{\ZZ}{\mathcal Z}
\newcommand{\bE}{\overline{E}}
\newcommand{\bF}{\overline{F}}
\newcommand{\wF}{\widetilde{F}}
\newcommand{\hcF}{\widehat{\mathcal F}}
\newcommand{\bW}{\overline{W}}
\newcommand{\bcW}{\overline{\mathcal W}}
\newcommand{\tL}{\widetilde{L}}
\newcommand{\tW}{\widetilde{W}}
\newcommand{\DeltaF}{\Delta^{\mathcal F}}
\newcommand{\SSS}{S^{(3)}}
\newcommand{\hM}{\widehat{M}}
\newcommand{\hL}{\widehat{L}}
\newcommand{\hLam}{\widehat{\Lambda}}
\newcommand{\hg}{\widehat{g}}
\newcommand{\diam}{{\rm diam}}
\newcommand{\diag}{{\rm diag}}
\newcommand{\Jac}{{\rm Jac}}
\newcommand{\Plong}{{\rm Plong}}
\newcommand{\Tr}{{\rm Tr}}
\newcommand{\Conv}{{\rm Conv}}
\newcommand{\Ext}{{\rm Ext}}
\newcommand{\Spec}{{\rm Sp}}
\newcommand{\Isom}{{\rm Isom}\,}
\newcommand{\Supp}{{\rm Supp}\,}
\newcommand{\Grass}{{\rm Grass}}
\newcommand{\Hold}{{\rm H\ddot{o}ld}}
\newcommand{\Ad}{{\rm Ad}}
\newcommand{\ad}{{\rm ad}}
\newcommand{\e}{{\rm e}}
\newcommand{\s}{{\rm s}}
\newcommand{\pol}{{\rm pole}}
\newcommand{\hol}{{\rm hol}}
\newcommand{\Aut}{{\rm Aut}}
\newcommand{\End}{{\rm End}}
\newcommand{\Cl}{{\rm Cl}}
\newcommand{\Eu}{{\rm Eu}}
\newcommand{\Leb}{{\rm Leb}}
\newcommand{\Liouv}{{\rm Liouv}}
\newcommand{\Lip}{{\rm Lip}}
\newcommand{\Int}{{\rm Int}}
\newcommand{\HD}{{\rm HD}}
\newcommand{\cc}{{\rm cc}}
\newcommand{\grad}{{\rm grad}}
\newcommand{\proj}{{\rm proj}}
\newcommand{\PSL}{{\rm PSL}}
\newcommand{\mass}{{\rm mass}}
\newcommand{\dive}{{\rm div}}
\newcommand{\dist}{{\rm dist}}
\newcommand{\im}{{\rm Im}}
\newcommand{\re}{{\rm Re}}
\newcommand{\codim}{{\rm codim}}
\newcommand{\Map}{\longmapsto}
\newcommand{\vide}{\emptyset}
\newcommand{\tr}{\pitchfork}
\newcommand{\ssl}{\mathfrak{sl}}

\newenvironment{demo}{\noindent{\textbf{Proof.}}}{\quad \hfill $\square$}
\newenvironment{pdemo}{\noindent{\textbf{Proof of the proposition.}}}{\quad \hfill $\square$}
\newenvironment{IDdemo}{\noindent{\textbf{Idea of proof.}}}{\quad \hfill $\square$}

\def\to{\mathop{\rightarrow}}
\def\act{\mathop{\curvearrowright}}
\def\To{\mathop{\longrightarrow}}
\def\Sup{\mathop{\rm Sup}}
\def\Max{\mathop{\rm Max}}
\def\Inf{\mathop{\rm Inf}}
\def\Min{\mathop{\rm Min}}
\def\lims{\mathop{\overline{\rm lim}}}
\def\limi{\mathop{\underline{\rm lim}}}
\def\egal{\mathop{=}}
\def\dans{\mathop{\subset}}
\def\surj{\mathop{\twoheadrightarrow}}

\maketitle

\begin{abstract}
We consider a transversally conformal foliation $\F$ of a closed manifold $M$ endowed with a smooth Riemannian metric whose restriction to each leaf is negatively curved. We prove that it satisfies the following dichotomy. Either there is a transverse holonomy-invariant measure for $\F$, or the foliated geodesic flow admits a finite number of physical measures, which have negative transverse Lyapunov exponents and whose basin cover a set full for the Lebesgue measure. We also give necessary and sufficient conditions for the foliated geodesic flow to be partially hyperbolic in the case where the foliation is transverse to a projective circle bundle over a closed hyperbolic surface.
\end{abstract}

\section*{Introduction}\label{intro}

The existence of a transverse holonomy-invariant measure for a foliation whose leaves have dimension $2$ or more is a rare phenomenon. The ergodic study of a foliation classically refers to the statistical description of Brownian paths tangent to its leaves: see \cite{Garnett}. In this paper we develop a different viewpoint and study the ergodic properties of geodesics tangent to the leaves of foliations.

All along this work $(M,\F)$ stands for a smooth (i.e. of class $C^{\infty}$) \emph{closed foliated manifold} of codimension $q$ endowed with a smooth Riemannian metric $g$. Up to passsing to a double cover, we will always assume that our foliations are \emph{oriented}. We will make the two following hypotheses

\begin{enumerate}
\item every leaf $L$ has \emph{negative sectional curvature} for the induced metric $g_{|L}$;
\item the foliation $\F$ is \emph{transversally conformal}.
\end{enumerate}

The first hypothesis is satisfied for example by every foliation transverse to a fiber bundle over a closed negatively curved manifold (see \S \ref{suspensioncocycles}). Moreover for every foliation by surfaces without transverse holonomy-invariant measure there exists a Riemannian metric on the ambient space such that the first hypothesis is satisfied (see Theorem \ref{negativeurvature}). The second hypothesis is satisfied by every codimension $1$ foliation and by (singular) holomorphic foliations on complex surfaces. It means that the \emph{holonomy pseudogroup} of $\F$ consists of conformal local diffeomorphisms of $\R^q$ (i.e. their derivatives at every point are similitudes of the Euclidean space).

We shall denote by $\hM$ the unit tangent bundle of the foliation $\F$ i.e. the set of unit vectors tangent to $\F$. Unit tangent bundles of leaves of $\F$ form  a foliation of $\hM$ denoted by $\hcF$. The \emph{foliated geodesic flow} is the smooth and leaf-preserving flow of $\hcF$ denoted by $G_t$ which induces on each leaf its geodesic flow. Since the leaves are negatively curved $G_t$ exhibits a weak form of hyperbolicity called \emph{foliated hyperbolicity} which is defined and analyzed by Bonatti, G{\'o}mez-Mont and Mart{\'i}nez in \cite{BGM}, and studied by the first author in \cite{AlvarezHarmonic,AlvarezGibbs,Alvarezu-Gibbs}. It means that there exist two continuous and $G_t$-invariant subfoliations of $\hcF$, called \emph{stable and unstable foliations} and denoted by $\W^s$ and $\W^u$, whose leaves are respectively uniformly exponentially contracted and expanded by $G_t$. This notion resembles the classical definition of \emph{partial hyperbolicity} for flows, which will be defined in \S \ref{phflows}, the transverse direction of the foliation playing the role of the central direction. But there is a main difference: the contraction, or expansion, in the transverse direction does not need to be dominated by the hyperbolicity inside the leaves. We will later discuss this matter.

The goal of that paper is twofold. Firstly we wish to describe the ergodic properties of the flow. Secondly we wish to discuss the relations between partial hyperbolicity and foliated hyperbolicity through the study of special examples: the foliations transverse to a projective circle bundle over a surface.

\paragraph{Finiteness of SRB measures --} Recall that an \emph{SRB measure} or \emph{physical measure} for $G_t$ is a $G_t$-invariant probability measure $\mu$ whose basin (the set of $v\in\hM$ such that the averages of Dirac masses along the orbit of $v$ converges to $\mu$ in the weak$^{\ast}$ sense) has positive Lebesgue measure.

They are named after by Sina{\u\i}, Ruelle and Bowen who introduced them for uniformly hyperbolic dynamics in \cite{Sinai,Ruelle,BR}. The question of the existence and finiteness of SRB measures for partially hyperbolic dynamics was studied by Bonatti, Viana in \cite{BV} and together with Alves in \cite{ABV}. It is proven in \cite{BV} that a partially hyperbolic diffeomorphism which is ``mostly contracting'' in the center direction has a finite number of SRB measures and that the union of their basins is full for the Lebesgue measure. Our situation is similar and we propose to prove the following dichotomy.

\begin{maintheorem}
\label{SRBmeasuresfgf}
Let $\F$ be a smooth transversally conformal foliation of a closed manifold $M$. Assume that $M$ is endowed with a smooth Riemannian metric such that every leaf is negatively curved for the restricted metric. Then we have the following dichotomy
\begin{itemize}
\item either there exists a transverse measure invariant by holonomy;
\item or $\hcF$ has a finite number of minimal sets each of which supports a unique SRB measure for $G_t$. These measures have negative transverse Lyapunov exponents and the union of their basins is full for the Lebesgue measure.
\end{itemize}
In the latter case it follows in particular that these SRB measures for $G_t$ are the unique ones and that $\F$ has finitely many minimal sets as well.
\end{maintheorem}

We define the \emph{transverse Lyapunov exponent} of an ergodic $G_t$-invariant measure $\mu$ as the following limit independent of the choice of a $\mu$-typical $v\in\hM$

$$\lambda^{\tr}(\mu)=\lim_{t\to\infty}\frac{1}{t}\log||D_v\h_{G_{[0,t]}(v)}\omega||,$$
where $\omega\in\NN_{\F}=TM/T\F$, the normal bundle of $\F$, and $\h_{G_{[0,t]}(v)}$ denotes the holonomy map along the orbit segment $G_{[0,t]}(v)$. The fact that the number is independent of $\omega$ (i.e. that there is a unique transverse exponent) follows from the condition that $\F$ is transverally conformal (see \S \ref{transverselyapunovsection}). The question of knowing what happens when there is more than one transverse Lyapunov exponent (when the foliation is not transversally conformal) seems a difficult one. Bonatti, Eskin and Wilkinson treat the case of foliations transverse to a projective $\C\PP^{q}$-fiber bundle over a closed hyperbolic surface 
in \cite{BEW}.

This dichotomy for foliations is reminiscent of a whole series of works initiated by Furstenberg (see for example \cite{Furst,Ledexp,KleptsynNalski}) and culminating with Avila-Viana's invariance principle \cite{AvilaViana}, which develops the following general principle. When composing randomly homeomorphisms of certain manifolds (for example the circle) either a probability measure is globally preserved, or there is some contraction in the dynamics.

Deroin and Kleptsyn were the first to include foliations inside the family of dynamical systems exhibiting this feature. In a wonderful paper \cite{DeroinKleptsyn}, which motivates the present paper, they proved Theorem \ref{SRBmeasuresfgf} with Garnett's \emph{harmonic measures} for $\F$ (see \cite{Garnett}) instead of SRB measures and they do not need the assumption on the sectional curvatures of the leaves. We also mention Forn{\ae}ss-Sibony's work on harmonic currents of laminations \cite{FoS}. In \cite{BGV} Bonatti, G{\'o}mez-Mont and Mart{\'i}nez showed how, in the case of foliations by \emph{hyperbolic manifolds}, to deduce Theorem \ref{SRBmeasuresfgf} from Deroin-Kleptsyn's result and from the bijective correspondence between harmonic measures and a special class of invariant measures called Gibbs $u$-states that we shall define below (see \cite{AlvarezHarmonic,BakhtinMartinez,Martinez}). Our proof of Theorem \ref{SRBmeasuresfgf} is independent of the study of the foliated Brownian motion and uses Pesin's theory of nonuniformly hyperbolic dynamical systems.

\paragraph{Foliations by surfaces --} It is not clear that for a given foliation there must exist an ambient Riemannian metric inducing negatively curved metrics in the leaves.

However, Ghys showed us a nice argument in order to prove that such a phenomenon is quite common in the world of two-dimensional foliations. Since it does not appear in the literature we propose to give in Appendix the proof the following result that we attribute to Ghys.

\begin{maintheorem}[Ghys]
\label{negativeurvature}
Let $(M,\F)$ be a closed manifold foliated by hyperbolic Riemann surfaces and $g$ be a smooth Riemannian metric on $M$. Then there exists a Riemannian metric $g'$ conformally equivalent to $g$ such that the Gaussian curvature of every leaf $L$ is negative for the restricted metric $g_{|L}'$.
\end{maintheorem}

We shall define in Appendix what is a foliation by hyperbolic surfaces. Many examples may be found in \cite{ADMVfol}. Let us emphasize that every two-dimensional foliation without transverse invariant measure must be a foliation by hyperbolic surfaces: see Proposition \ref{hyptype}. As a consequence we deduce that for every transversally conformal two-dimensional foliation the following dichotomy holds true.

\begin{itemize}
\item Either it has a transverse holonomy-invariant measure.
\item Or there exists a smooth Riemannian metric of the ambient space such that every leaf is negatively curved for the restricted metric. For such a metric the second alternative of Theorem \ref{SRBmeasuresfgf} holds true.
\end{itemize}

\paragraph{Partially hyperbolic examples --} We now raise the problem of the relation between partial and foliated hyperbolicities. We illustrate it by a detailed study of special examples, namely foliations transverse to a circle bundle over a hyperbolic surface with projective holonomy. We propose a link between partial hyperbolicity of the foliated geodesic flow and a purely topological condition on the bundle: the value of its Euler number.

Recall that circle bundles over a closed surface $\Sigma$ of genus $\g$ are classified by an integer, their \emph{Euler number}, and that those admitting a transverse foliation are precisely those whose Euler number is, in absolute value, less than $2\g-2$ (this is Milnor-Wood's inequality for which we refer to \cite{Milnor,Wood}). The Euler number of a circle bundle $\Pi:M\to\Sigma$ is denoted by $\Eu(\Pi)$.

A smooth foliation $\F$ transverse to a circle bundle $\Pi:M\to\Sigma$ is obtained from its \emph{holonomy representation} $\hol:\pi_1(\Sigma)\to\dif^{\infty}(S^1)$ by a process called \emph{suspension} (see \cite{CL}).  When the foliated bundle has a \emph{projective holonomy group} (i.e. each fibers is identified with the real projective line $\R\PP^1$ and the holonomy representation takes its values in the group $\PSL_2(\R)$ of projective transformations of $\R\PP^1$) we say that the data $(\Pi,M,\Sigma,\F)$ is a \emph{foliated $\R\PP^1$-bundle with projective holonomy}. 

Let $(\Pi,M,\Sigma,\F)$ be a foliated $\R\PP^1$-bundle with projective holonomy and suppose $\Sigma$ is endowed with a hyperbolic metric $m$. Say a smooth Riemannian metric $g$ on $M$ is \emph{admissible} if for every leaf $L$ the restriction $\Pi_{|L}:(L,g_{|L})\to(\Sigma,m)$ is a Riemannian cover. In that case $G_t$ preserves the fibers of the bundle $\Pi_{\star}:\hM\to T^1\Sigma$ induced by the differential of $\Pi$, and the fiber direction, which as we shall prove is invariant by the flow, is a good candidate for being the central direction.

Our next result provides a topological condition on a circle foliated bundle with projective holonomy to possess a partially hyperbolic foliated geodesic flow (we refer to \S \ref{phflows} for the definition of partially hyperbolic flows). This provides new geometric examples of partially hyperbolic dynamical systems.

\begin{maintheorem}
\label{domination}
Let $\Sigma$ be a closed surface of genus $\g\geq 2$.

\begin{enumerate}
\item If $(\Pi,M,\Sigma,\F)$ is a foliated $\R\PP^1$-bundle with projective holonomy satisfying $|\Eu(\Pi)|<2\g-2$ then there exists a hyperbolic metric on $\Sigma$ such that for every admissible Riemannian metric on $M$ the foliated geodesic flow of $\F$ is partially hyperbolic.
\item For every hyperbolic metric on $\Sigma$ there exists a foliated $\R\PP^1$-bundle with projective holonomy such that for every admissible Riemannian metric on $M$ the foliated geodesic flow $G_t$ is partially hyperbolic. Moreover $\Eu(\Pi)$ can be made arbitrary in $\{3-2\g,...,0,...,2\g-3\}$.
\item If $(\Pi,M,\Sigma,\F)$ is a foliated $\R\PP^1$-bundle with projective holonomy satisfying $|\Eu(\Pi)|=2\g-2$ and $\Sigma$ is endowed with any hyperbolic metric then for every admissible Riemannian metric on $M$ the foliated geodesic flow of $\F$ is not partially hyperbolic.
\end{enumerate}
\end{maintheorem}

This result is a consequence of a result coming from the field of $3$-dimensional Anti-de Sitter geometry proven recently and independently by Gu{\'e}ritaud-Kassel-Wolff in \cite{GKW} and Deroin-Tholozan in \cite{DeroinTholozan}, as well as Theorem \ref{equivalencedomination} which shall be stated and proven in \S \ref{suspensioncocycles}

\paragraph{Outline of the work --} In Section \ref{sectionpreliminaries} we give the necessary material which will be used throughout the text. Section \ref{mostly_contracting} is devoted to the proof of Theorem \ref{negativeexponent}, the main technical result used to prove Theorem \ref{SRBmeasuresfgf}. In Section \ref{finiteness} we finish the proof of Theorem \ref{SRBmeasuresfgf}. Section \ref{phfoliatedgeoflows} is devoted to the presentation of the notion of domination of projective representation and to proving Theorem \ref{equivalencedomination}, the main technical result used to prove Theorem \ref{domination}. Fuchsian foliations are treated in Section \ref{fuchsianfoliations} and Theorem \ref{transversefuchsian}, which gives the transverse Lyapunov exponent of the unique SRB measure in that case, is stated and proven there. In Appendix, we give Ghys' argument for proving Theorem \ref{negativeurvature}.

\section{Preliminaries}\label{sectionpreliminaries}

\subsection{Transversally conformal foliations}

\paragraph{Notations and definitions --}Let $\F$  be a smooth \emph{foliation} of a smooth closed manifold $M$ with codimension $q$. Denote by $\Pp$ the \emph{holonomy pseudogroup} of $\F$ associated to a foliated atlas. It consists of local diffeomorphisms of $\R^q$.

The set of vectors tangent to $\F$ is a subbundle $T\F\dans TM$ called the \emph{tangent bundle} of $\F$. The \emph{normal bundle} of $\F$ is by definition $\NN_{\F}=TM/T\F$. The choice of a smooth Riemannian metric of $M$ identifies $\NN_{\F}$ with $T\F^{\perp}$.

\paragraph{Transversally conformal foliations --} Say the foliation $\F$ is \emph{transversally conformal} if elements of $\Pp$ are conformal transformations of $\R^q$, meaning that their derivatives are everywhere similitudes of the Euclidian space.

If one prefers that means that there exists $|.|$, a \emph{transverse metric} for $\F$, such that for every path $c$ tangent to $\F$, every $x$ inside the domain of $\h_c$, a holonomy map along $c$, and every $v\in\NN_{\F}(x)$, one has
$$|D_x\h_c(v)|=\lambda|v|,$$
for some $\lambda>0$ independent of $v$.

\begin{rem}
Every codimension $1$ foliation is transversally conformal.
\end{rem}

\subsection{Invariant measures}
\label{subsec_invariant}

\paragraph{Transverse  holonomy-invariant measures --} Let $(T_i)_{i\in I}$ be a complete system of transversals to the foliation, i.e. a finite family of transversals to $\F$ whose union meets every leaf. A \emph{transverse holonomy-invariant measure} (or simply transverse invariant measure) is a family of finite nonnegative measures $(\nu_i)_{i\in I}$ satisfying

\begin{enumerate}
\item $\nu_i(T_i)>0$ for some $i\in I$;
\item if for $i,j\in I$ there is a holonomy map $\h:S_i\to S_j$ between two open sets $S_i\dans T_i$ and $S_j\dans T_j$, then for any Borel set $A_i\dans S_i$ we have $\nu_i(A_i)=\nu_j(\h(A_i))$.
\end{enumerate}

\paragraph{Totally invariant measures --} The Riemannian metric induces a Riemannian structure on each leaf. Hence every leaf is endowed with a natural volume form. Assume that the foliation $\F$ possesses a transverse invariant measure $(\nu_i)_{i\in I}$. Then, if $\nu_i(T_i)>0$, it is possible inside a corresponding foliated chart $U_i$ to integrate the volume of the plaques (i.e. the connected components of the intersections of leaves of $\F$ with $U_i$) against $\nu_i$. We obtain this way a measure $m_i$ in the chart $U_i$.

The family of measures $(\nu_i)_{i\in I}$ is holonomy-invariant, so the $m_i$ glue together and provide a finite measure $m$ on $M$. Such a measure will be from now one called \emph{totally invariant}.

\subsection{The foliated geodesic flow and foliated hyperbolicity}
In what follows we assume the existence of a smooth Riemannian metric $g$ on $M$ such that the sectional curvature of every leaf $L$ for the restricted metric $g_{L}$ is negative. By compactness of $M$ this implies that the sectional curvatures of every leaf $L$ are uniformly pinched between two negative constants $-b^2<-a^2<0$.

\subsubsection{The foliated geodesic flow}
\paragraph{Unit tangent bundle --} Let $\hM$ denote the \emph{unit tangent bundle of $\F$}, i.e. the subbundle of $T^1M$ consisting of those unit vectors tangent to $\F$. It is a closed manifold endowed with a smooth foliation denoted by $\hcF$ whose leaves are the unit tangent bundles of leaves of $\F$. We will denote by $pr:\hM\to M$ the canonical \emph{basepoint projection} associating its basepoint to each vector $v\in\hM$. The Sasaki metric induces a metric on $\hM$ by $\hg$. We will denote by $\hL_v$ the leaf of $v\in\hM$ i.e. $T^1L_x$ where $L_x$ is the leaf of the basepoint $x$ of $v$. Note that the foliations $\F$ and $\hcF$ have the same holonomy pseudogroups (see  \cite[\S 3.1]{Alvarezu-Gibbs}) so the following proposition holds.

\begin{proposition}
\label{holFhcF}
There exists a transverse invariant measure for $\F$ if and only if there exists one for $\hcF$. The foliation $\F$ is transversally conformal if and only if $\hcF$ is transversally conformal.
\end{proposition}

\paragraph{Foliated geodesic flow --} We call \emph{foliated geodesic flow} the leaf-preserving flow of $(\hM,\hcF)$ which induces on each leaf $\hL$ its geodesic flow. This flow will be denoted by $(t,v)\in\R\times\hM\mapsto G_t(v)$. It is smooth since the leafwise geodesic equation depends only on the $1$-jet of the restricted metric $g_L$.

\subsubsection{Foliated hyperbolicity}

When all sectional curvatures are negative the foliated geodesic flow exhibit a weak form of hyperbolicity called \emph{foliated hyperbolicity} in \cite{BGM}. Following \cite[Chapter IV]{Ba},  one proves that there are two continuous and $DG_t$-invariant subbundles of $T\hcF$ of the same dimension $p-1$ ($p$ being the dimension of $\F$) denoted by $E^s$ and $E^u$, which satisfy $T\hcF=E^s\oplus\R X\oplus E^u$ ($X$ being the generator of $G_t$) and are respectively uniformly contracted and expanded. They are called \emph{stable and unstable bundles}.

As explained in \cite{Alvarezu-Gibbs,BGM} the usual \emph{Stable Manifold Theorem} applies in that case and these bundles are uniquely integrable. We denote by $\W^s$ and $\W^u$ the two continuous and $G_t$-invariant subfoliations of $\hcF$ tangent to the distributions $E^s$ and $E^u$.  We call them the \emph{stable and unstable foliations}. Their leaves are denoted by  $W^s(v)$ and $W^u(v)$ and are called stable and unstable manifolds. The saturations of $\W^s$ and $\W^u$ by $G_t$ form two continuous subfoliations of $\hcF$ tangent to $E^{cs}$ and $E^{cu}$ called \emph{center-stable and center-unstable foliations} and denoted by $\W^{cs}$ and $\W^{cu}$. The leaves of $\W^{cs}$ and $\W^{cu}$ passing through $v\in\hM$ are denoted by $W^{cs}(v)$ and $W^{cu}(v)$ and called respectively the \emph{center-stable} and \emph{center-unstable manifolds} of $v$.

\subsection{Lyapunov exponents}

Many bundles are involved in the theory of foliated hyperbolicity. Hence we found useful to include a detailed discussion about Oseledets' splitting. In particular we want to prove that the transverse Lyapunov exponent (see \S \ref{transverselyapunovsection}) is a  ``classical'' Lyapunov exponents for $G_t$.

\subsubsection{Transverse Lyapunov exponent}
\label{transverselyapunovsection}

The linear holonomy over the foliated geodesic flow defines a linear cocycle on $\NN_{\hcF}$ over $G_t$, so we can apply Oseledets' theorem to that cocycle. Using that the foliation $\hcF$ is transversally conformal (see Proposition \ref{holFhcF}) we obtain the following

\begin{proposition}
There exists a Borel set $\X_0\dans\hM$ which is $G_t$-invariant and full for every $G_t$-invariant measure such that for every $v\in\X_0$ and $\omega\in \NN_{\hcF}(v)$ the following number is well defined and idependent of $\omega$ and of the transverse metric $|.|$
$$\lambda^{\tr}(v)=\lim_{t\to\infty}\frac{1}{t}\log |D_v\h_{G_{[0,t]}(v)}\omega|$$
where $\h_{G_{[0,t]}(v)}$ denotes the holnomy map along the orbit segment $G_{[0,t]}(v)$.
\end{proposition}

\subsubsection{Oseledets' splitting}
\label{oselouille}

\paragraph{Oseledets' theorem --} Considering the linear cocycle given by the derivative of $G_1:\hM\to\hM$ we find a Borel set $\X_1\dans\hM$ full for every $G_t$-invariant measure such that for every $v\in\X_1$ there exists a splitting
\begin{equation}
\label{Oseledetssplitting}
T_v\hM=E_1(v)\oplus...\oplus E_k(v)
\end{equation}
which is $DG_t$-invariant and such that for every $i\in\{1,...,k\}$ and every $\omega\in E_i(v)$ the \emph{Lyapunov exponent}
$$\lambda_i(v)=\lim_{t\to\infty}\frac{1}{t}\log||D_vG_t \omega||$$
is well defined.

\paragraph{Stable and unstable Lyapunov exponents --} The bundles $E^s$ and $E^u$ are $DG_t$-invariant. Oseledets' theorem applied to the restriction of $DG_1$ to these bundles provides a Borel set $\X_2\dans\X_1$ full for every $G_t$-invariant measure such that for every $v\in\X_2$ there exist two splittings
$$E^s(v)=E_1^s(v)\oplus....\oplus E_{k_s}^s(v)$$
$$E^u(v)=E_1^u(v)\oplus...\oplus E_{k_u}^u(v)$$
which are $DG_t$-invariant and such that, when $\star=u$ or $s$, for every $i\in\{1,...,k_{\star}\}$ and $\omega\in E_i^{\star}(v)$ the Lyapunov exponent
$$\lambda_i^{\star}(v)=\lim_{t\to\infty}\frac{1}{t}\log||D_vG_t \omega||$$
is well defined.

Note that these numbers belong to the family of Lyapunov exponents $(\lambda_1(v),...,\lambda_k(v))$. Moreover the ``tangential'' Lyapunov exponents at $v\in\X_2$ of $G_t$ (i.e. those of the linear cocycle $(DG_t)_{|T\hcF}$) are precisely $(\lambda_1^s(v),...,\lambda_{k_s}^s(v),0,\lambda_1^u(v),...,\lambda_{k_u}^u(v))$, where $0$ corresponds to the flow direction.

Moreover one has $\lambda_i^s(v)<0$ and $\lambda_j^u(v)>0$ for every $v\in\X_2$, $i\in\{1,...k_s\}$ and $j\in\{1,...,k_u\}$. Denote for $v\in\X_2$
\begin{equation}
\label{stablesumlyap}
\Lambda^s(v)=\sum_{i=1}^{k_s}\dim E_i^s(v)\,\,\,\lambda^s_i(v) <0
\end{equation}
and
\begin{equation}
\label{unstablesumlyap}
\Lambda^u(v)=\sum_{j=1}^{k_u}\dim E_j^u(v)\,\,\,\lambda^u_j(v) >0.
\end{equation}
Since the foliated geodesic flow preserves the Liouville measure inside the leaves it follows that
\begin{equation}
\label{sumstableunstable}
\Lambda^s(v)+\Lambda^u(v)=0.
\end{equation}

\subsubsection{Lyapunov spectrum}
\label{spectrouille}

\paragraph{Linear Poincar{\'e} flow --} Recall that $X$ denotes the generator of the foliated geodesic flow. This vector field is everywhere nonzero so we can define the normal bundle $\NN_X=X^{\perp}$. The \emph{linear Poincar{\'e} flow} $\Psi_t$ defines a linear cocycle of $\NN_X$ over $G_t$ (see \cite[\S 2.6]{ArPa}). Recall that for every $v\in\hM$, we have
$$\Psi_t(v)=\pi_{G_t(v)}\circ D_v G_t:\NN_X(v)\to\NN_X(G_t(v)),$$
where $\pi_v:T_v\hM\to\NN_X(v)$ denotes the orthogonal projection.

Oseledets' theorem applies for that cocycle and provides a Borel set $\X_3\dans\X_2$ full for every $G_t$-invariant measure such that for every $v\in\X_3$ there is a splitting
$$\NN_X(v)=F_1(v)\oplus...\oplus F_l(v)$$
which is $\Psi_t$-invariant and such that for every $i\in\{1,...,l\}$ and every $\omega\in\NN_X(v)$ the Lyapunov exponent
$$\chi_i(v)=\lim_{t\to\infty}\frac{1}{t}\log||\Psi_t(v)\omega||$$
is well defined. Moreover it is noted in \cite[\S 2.7.2.1]{ArPa} that since the variation of the angles between subspaces defining Oseledets splitting \eqref{Oseledetssplitting} is subexponential along the orbits of $G_t$, the Lyapunov exponents $\chi_i$ coincide with the Lyapunov exponents $\lambda_j$ \emph{off the flow direction} and the space $F_i$ coincide with the orthogonal projection to $\NN_X$ of some space $E_j$.

\paragraph{Lyapunov spectrum --} Now we are ready to conclude our discussion and to identify the Lyapunov spectrum of the foliated geodesic flow.

The metric $\hg$ identifies $\NN_{\hcF}$ and $T\hcF^{\perp}$. Note that since $X$ is tangent to $\hcF$ we have $\NN_{\hcF}\dans\NN_X$ and we can decompose $\NN_{X}=\NN_{\hcF}\oplus\NN'$ where $\NN'=\NN_X\cap T\hcF$ is a $\Psi_t$-invariant bundle. Lyapunov exponents of the cocycle $(\Psi_t)_{|\NN'}$ are precisely the stable and unstable Lyapunov exponents defined above. Define
$$\Psi^{\hcF}_t(v)=\pi_{G_t(v)}^{\hcF}\circ\Psi_t(v):\NN_{\hcF}(v)\to\NN_{\hcF}(G_t(v))$$
and argue as in the previous paragraph in order to get that the Lyapunov exponents of that cocycle are precisely the Lyapunov exponents of $\Psi_t$ \emph{off the $\NN'$ direction}.

Finally remark that the latter cocycle coincides with the linear holonomy over the foliated geodesic flow, i.e. for every $v\in\hM$
$$\Psi^{\hcF}_t(v)=D_v\h_{G_{[0,t]}(v)}$$
and hence that cocycle has only one Lyapunov exponent which is the \emph{transverse Lyapunov exponent} defined in \S \ref{transverselyapunovsection}. Hence the following proposition follows from our discussion.

\begin{proposition}
\label{Lyapspectrum}
There exists a Borel set $\X\dans\hM$ full for every  $G_t$-invariant measure such that for every $v\in\X$ the Lyapunov exponents of $G_t$ at $v$ exist and are precisely given by the list 
$$(\lambda^s_1(v),...,\lambda^s_{k_s}(v),0,\lambda_1^u(v),...,\lambda^u_{k_u}(v),\lambda^{\tr}(v)),$$
where a priori $\lambda^{\tr}$ can be equal to $\lambda^s_i,0$ or $\lambda^u_j$.
\end{proposition}

\begin{defi}
We define the transverse Lyapunov exponent of a $G_t$-invariant measure $\mu$ as the integral
$$\lambda^{\tr}(\mu)=\int_{\hM}\lambda^{\tr}(v)\,d\mu(v).$$
The average sums $\Lambda^s(\mu)$ and $\Lambda^u(\mu)$ are defined similarly.
\end{defi}

\subsection{Gibbs $u$-states}
\label{defgubbsustates}

\paragraph{Definition --} A \emph{Gibbs $u$-state} for $G_t$ is a $G_t$-invariant probability measure on $\hM$ whose conditional measures in local unstable manifolds are equivalent to the Lebesgue measure.

A \emph{Gibbs $s$-state} for $G_t$ is a Gibbs $u$-state for $G_{-t}$: its conditional measures in local stable manifolds are equivalent to the Lebesgue measure.

A \emph{Gibbs $su$-state} is a probability measure on $\hM$ which is both a Gibbs $s$-state and a Gibbs $u$-state.

Gibbs $u$-states were first introduced by Pesin-Sina{\u\i} in \cite{PS} and then by Bonatti-Viana in \cite{BV} in the context of partially hyperbolic dynamics.

\paragraph{Example --} Suppose $\F$, and therefore $\hcF$ as well, has a family of transverse invariant measures $(\nu_i)_{i\in I}$. As explained in \S \ref{subsec_invariant} this measure can be combined with the Liouville measure in the plaques of $\hcF$ so as to construct a totally invariant measure $\mu$ on $\hM$.

Since $G_t$ preserves the Liouville measure of the leaves it must  preserve $\mu$. Moreover by absolute continuity of $\W^s$ and $\W^u$ (see \cite{Alvarezu-Gibbs}) it must have Lebesgue disintegration in both stable and unstable plaques. \emph{Such a measure is a Gibbs $su$-state.}

\paragraph{Existence --} Define the unstable Jacobian at $v\in\hM$ by $\Jac^u G_t(v)=\det (D_vG_t)_{|E^u(v)}.$ The proof of the next theorem follows the line of reasoning of \cite[Section 11.2.2]{BDV}. Recall that we say that a positive function (resp. a family of positive functions) is \emph{log-bounded} (resp. \emph{uniformly log-bounded}) if it is bounded away (resp. uniformly bounded away) from $0$ and $\infty$ (i.e. its logarithm is bounded).

\begin{theorem}
\label{reconstruirelesetatsdeugibbs}
Let $(M,\F)$ be a closed manifold endowed with a smooth Riemannian metric $g$ such that every leaf $L$ is negatively curved for the induced metric $g_{|}L$. Then

\begin{enumerate}
\item for every $v\in\hM$ and every Borel set $D^u\dans W^u_{loc}(v)$ with positive Lebesgue measure, any accumulation point of the following family of measures, indexed by $T\in(0,\infty)$, is a Gibbs u-state
$$\mu_T=\frac{1}{T}\int_0^T G_{t\ast}\left(\frac{\Leb^u_{|D^u}}{\Leb^u(D^u)}\right)dt.$$
Moreover its densities along unstable plaques, denoted by $\psi^u_w$, are uniformly log-bounded and satisfy the following for $z_1,z_2\in W^u_{loc}(w)$
\begin{equation}
\label{relationdensiteslocales}
\frac{\psi^u_w(z_2)}{\psi^u_w(z_1)}=\lim_{t\to\infty}\frac{\Jac^u G_{-t}(z_2)}{\Jac^u G_{-t}(z_1)};
\end{equation}
\item all ergodic components of a Gibbs u-state are Gibbs u-states with local densities in the unstable plaques that are uniformly log-bounded and satisfy \eqref{relationdensiteslocales};
\item every Gibbs u-state for $G_t$ is a measure whose local densities in the unstable plaques are uniformly log-bounded and satisfy \eqref{relationdensiteslocales}.
\end{enumerate}
\end{theorem}

\section{Transversally mostly contracting foliated geodesic flows}\label{mostly_contracting}
This section is devoted to proving the next theorem. We will then follow the ``mostly contracting scenario'' developped by \cite{BV} for partially hyperbolic systems, to prove Theorem \ref{SRBmeasuresfgf} (see also \cite{BGM}).

\begin{maintheorem}
\label{negativeexponent}
Let $\F$ be a smooth transversally conformal foliation of a closed manifold $M$. Assume that $M$ is endowed with a smooth Riemannian metric such that every leaf $L$ is negatively curved for the restriction $g_{L}$. Assume that there is no transverse holonomy invariant measure for $\F$. Then the transverse Lyapunov exponent of every Gibbs u-state is negative.
\end{maintheorem}

\begin{rem}
Let $\mu$ be a totally invariant measure (it is a Gibbs $u$-state). It is easy to prove that $\lambda^{\tr}(\mu)=0$. For that purpose, one uses the map $\iota:v\in\hM\mapsto -v$ and note that $\lambda^{\tr}(v)=-\lambda^{\tr}(\iota(v))$ for $\mu$-almost every $v$. This map preserves $\mu$ (since it preserves the leaves and the Liouville measure in the leaves: see \cite[Lemma 1.34]{Paternain}) so, by integrating the last equality against $\mu$, we find $\lambda^{\tr}(\mu)=-\lambda^{\tr}(\mu)$.
\end{rem}

\subsection{Existence of transverse invariant measures: proof of Theorem \ref{negativeexponent}} 
\label{proofthmB}

\paragraph{The main criterion --} The following criterion for the existence invariant measure was proven in \cite[Theorem A]{Alvarezu-Gibbs}.

\begin{theorem}
\label{sugibbs}
Let $(M,\F)$ be a closed foliated manifold by negatively curved manifolds. Then every Gibbs $su$-state for the foliated geodesic flow $G_t$ as defined in \S \ref{defgubbsustates} is totally invariant. In particular if such a measure exists, $\hcF$ and $\F$ both possess a transverse invariant measure.
\end{theorem}

Starting with a Gibbs $u$-state $\mu$ with $\lambda^{\tr}(\mu)\geq 0$, we want to use the above criterion and to prove that it has Lebesgue disintegration along $\W^s$. 

In order to do so, our strategy is to prove that the inverse flow satisfies Pesin's entropy formula \ref{Pesinformula} and finally to conclude using Ledrappier-Young's work \cite{LYI}.

\paragraph{Entropy of Gibbs $u$-states --} The first step is to prove the following general proposition about metric entropy of Gibbs $u$-states. Such a statement may be found in \cite{LedStr} in a much more general context. In ours the proof is much shorter and shall be postponed until the end of the section.

\begin{proposition}
\label{minentropy}
Let $\mu$ be an ergodic Gibbs $u$-state for $G_t$. Then
$$h_{\mu}\geq\Lambda^u(\mu),$$
where $h_{\mu}$ denotes the metric entropy of $G_t$ for the measure $\mu$.
\end{proposition}

\paragraph{Proof of Theorem \ref{negativeexponent} --}  Let us assume the existence of $\mu$, a Gibbs $u$-state for $G_t$ whose transverse Lyapunov exponent is nonnegative. By Theorem \ref{reconstruirelesetatsdeugibbs}, ergodic components of Gibbs $u$-states are Gibbs $u$ states, so we can assume that $\mu$ is ergodic.

From now on $\mu$ is supposed to be ergodic. Since it has a nonnegative transverse Lyapunov exponent it follows from Proposition \ref{Lyapspectrum} that the only negative Lyapunov exponents of $\mu$ are the Lyapunov exponents in the stable direction $\lambda^s_i(\mu)$. Then Ruelle's inequality (see \cite{Ruelleineq}) applied to $G_{-t}$ (which has the same measure entropy as $G_t$) implies that
$$h_{\mu}\leq-\Lambda^s(\mu).$$

Remember that since $G_t$ preserves the Liouville measure of the leaves, we have $\Lambda^u(\mu)=-\Lambda^s(\mu)$. Hence using Proposition \ref{minentropy}, we see that the following \emph{Pesin's formula} holds true

\begin{equation}
\label{Pesinformula}
h_{\mu}=-\Lambda^s(\mu).
\end{equation}

By Ledrappier-Young (see \cite{LYI}), this equality holds if and only if $\mu$ is a Gibbs $u$-state for $G_{-t}$, or if one prefers, a Gibbs $s$-state for $G_t$.

As a consequence we find that if $\mu$ has nonnegative transverse Lyapunov exponent, then it is a Gibbs $su$-state and, by Theorem \ref{sugibbs}, has to be totally invariant, i.e. locally the product of the Liouville measure by a transverse invariant measure, thus concluding the proof of Theorem \ref{negativeexponent}.

\subsection{Proof of Proposition \ref{minentropy}}

In what follows, $\mu$ is an ergodic Gibbs $u$-state for the foliated geodesic flow.

\paragraph{Metric entropy according to Katok --}

Define \emph{Bowen's dynamical balls} $B_{t,r}(v)$ as follows. For $v\in\hM$,$t\geq 0$ and $r>0$, we say that $w\in B_{t,r}(v)$ if for every $s\in[0,t]$, $\dist(G_s(v),G_s(w))\leq r$.

Given $\delta\in(0,1)$, denote by $N(t,r,\delta)$ the minimal number of dynamical balls $B_{t,r}(v)$ needed to cover a set of measure greater than $1-\delta$. Katok proved in \cite{Katok} that we have for every $\delta\in(0,1)$

\begin{equation}
\label{katokentropy}
h_{\mu}=\lim_{r\to 0}\lims_{t\to\infty}\frac{\log\,N(t,r,\delta)}{t}.
\end{equation}

\paragraph{Approaching the Lyapunov exponents --} Oseledets' theorem applied to the cocycle $DG_{|E^u}:E^u\to E^u$ provides a set  $\X^u$ full for $\mu$ such that for every $v\in\X^u$, the following equality holds

$$\lim_{t\to\infty}\frac{1}{t}\log\,\Jac^u\,G_t(v)=\Lambda^u(\mu)=\Lambda^u.$$
Fixing $\eps>0$ and a positive time $t$, we shall define the following measurable set

\begin{equation}
\label{xute}
\X^u_{t,\eps}=\left\{v\in\X^u;\,\left|\frac{1}{t}\log\,\Jac^u G_t(v)-\Lambda^u\right|<\eps\right\},
\end{equation}
and notice that $\lim_{t\to\infty}\mu(\X^u_{t,\eps})=1$ for every $\eps>0$.

From now on, we fix a small $\eps>0$, which we will let tend to zero at the end of the proof.

\paragraph{Atlas for the unstable foliation --} Consider a finite atlas $(U_i,\phi_i)_{i\in I}$ for the unstable foliation such that $U_i$ reads as a union of local unstable manifolds $D^u(v_i)=W_{loc}^u(v_i)$.

Since the foliation $\W^u$ is continuous in the $C^{\infty}$-topology, we find universal bounds for diameters and volumes of the unstable plaques.

For every chart $U_i$ with $\mu(U_i)>0$, we can disintegrate $\mu$ in the local unstable manifolds of $U_i$. The conditional measures have local densities with respect to the volume which are uniformly log-bounded by a constant independent of $i$: see Theorem \ref{reconstruirelesetatsdeugibbs}.

\paragraph{Unstable volume of dynamical balls --} The following lemma is a useful consequence of the distortion controls.

\begin{lemma}
\label{distortionbowen}
\begin{enumerate}

\item There exists a constant $r_0>0$ such that for every $r<r_0$, $w\in\hM$, $t\geq 0$ and every measurable set $V\dans\hM$ contained in an unstable plaque containing $w$ we have
$$B_{t,r}(w)\cap V=G_{-t}(W_r^u(G_t(w)))\cap V.$$

\item There exists a positive constant $C_0>0$ such that for every $r<r_0$ and every unstable plaque $D^u=D^u(v_i)$ and every $w\in\hM$ we have
$$\Leb^u(B_{t,r}(w)\cap D^u)\leq C_0\frac{\Leb^u(W^u_r(G_t(w)))}{\Jac^u G_t(w)}.$$
\end{enumerate}
\end{lemma}

\begin{proof}
The first part of the lemma follows classically from the fact that $G_t$ expands uniformly the unstable foliation $\W^u$.

The second part follows from the first one, from the classical distortion control, as well as from the fact that unstable plaques are uniformly bounded.
\end{proof}

\paragraph{Covering the local unstable manifolds --} Fixing a small $\delta>0$, we get $T_0$ such that for every $t\geq T_0$, and every $i\in I$ with $\mu(U_i)>0$, the relative mass $\mu(^c\,\X_{t,\eps}^u\cap U_i)/\mu(U_i)$ is less that $\delta$ (we adopt the notation $^c\,X$ for the complement of a set $X$). In particular, the mass of $^c\,\X_{t,\eps}^u$ is less that some constant times $\delta$.

Now choose $r$ smaller than the $r_0$ given by Lemma \ref{distortionbowen} and cover $\X_{t,\eps}^u$ by $p$ dynamical balls $B_{t,r}(w_j)$, $w_1,...w_p\in\hM$. We can assume without loss of generality that all these dynamical balls intersect $\X_{t,\eps}^u$. Moreover, up to consider dynamical balls associated to $2r$ instead to $r$, which doesn't affect our argument, we can ask that all $w_j$ belong to $\X_{t,\eps}$. We now have two facts

\begin{itemize}
\item $\mu(^c\bigcup_{j=1}^p B_{t,r}(w_j)\cap U_i)\leq\delta\mu(U_i)$.
\item $\mu_{|U_i}$ has Lebesgue disintegration in the unstable plaques $D^u(x_i)$ with local densities which are \emph{uniformly log-bounded} (see Theorem \ref{reconstruirelesetatsdeugibbs}).
\end{itemize}

From this and Lemma \ref{distortionbowen}, we find a plaque $D^u$ and uniform constants $C_1,\delta_1>0$ such that

\begin{equation}
\label{majorer}
\Leb^u\left(\bigcup_{j=1}^p B_{r,t}(w_j)\cap D^u\right)\geq (1-\delta_1)\Leb^u(D^u).
\end{equation}

\begin{equation}
\label{minorer}
\Leb^u(B_{r,t}(w_j)\cap D^u)\leq C_1\frac{\Leb^u(W^u_r(G_t(w_j)))}{\Jac^u\,G_t(w_j)}.
\end{equation}

Putting \eqref{majorer} and \eqref{minorer} together and using a uniform upper bound for $\Leb^u(W^u_r(v))$, we get a constant $C_r>0$ depending only on $r$ such that

\begin{eqnarray*}
(1-\delta_1)\Leb^u(D^u)&\leq &\sum_{j=1}^{p}\Leb^u(B_{r,t}(w_j)\cap D^u)\\
                       &\leq & C_r\sum_{j=1}^p\frac{1}{\Jac^u\,G_t(w_j)}.
\end{eqnarray*}
Since $w_j\in\X^u_{t,\eps}$ for every $j$, we see that by definition of $\X_{t,\eps}^u$ (see \eqref{xute}) $\frac{1}{t}\Jac^u\,G_t(w_j)\geq\Lambda^u-\eps$. Finally there exists a constant $C_r'>0$ depending only on $r$ such that the following lower bound for $p$ holds

\begin{equation}
\label{lbp}
C_r'\exp\left[t(\Lambda^u-\eps)\right]\leq p.
\end{equation}

Lower bound \eqref{lbp} holds for every open cover by dynamical balls of $\X_{t,\eps}^u$, which is of measure $\geq 1-\delta$: it also provides a lower bound for $N(t,r,\delta)$. Finally, we find $\lims_{t\to\infty}t^{-1}\log N(t,r,\delta)\geq\Lambda^u-\eps$. Since $\eps$ is arbitrary, we deduce that $h_{\mu}\geq\Lambda^u(\mu)$ as desired. \quad \hfill $\square$

\section{Finiteness of SRB measures and of minimal sets}\label{finiteness}

Here we prove Theorem \ref{SRBmeasuresfgf}. It remains to prove that if every Gibbs $u$-state for $G_t$ has negative transverse Lyapunov exponent, then $G_t$ has finitely many SRB measures, their transverse Lyapunov exponents are negative and $\F$ has finitely many minimal sets.

\subsection{Finiteness of SRB measures}
\begin{proposition}
\label{srbugibbs}
Let $\F$ be a transversally conformal foliation of a closed manifold  $M$. Assume that $M$ is endowed with a smooth Riemannian metric such that every leaf is negatively curved. Assume that all Gibbs $u$-states have negative transverse Lyapunov exponent. Then the following assertions hold true.
\begin{enumerate}
\item Every ergodic Gibbs $u$-state is a SRB measure.
\item There is a finite number of Gibbs $u$-states.
\item The supports of the SRB measures are disjoint minimal sets of $\W^{cu}$.
\end{enumerate}
\end{proposition}

\begin{proof}
The proofs of Items 1. and 2. are given by Bonatti, G{\'o}mez-Mont  and Mart{\'i}nez in \cite[p.16-17]{BGM}. The idea is to use Pesin's stable manifold theory (which we will introduce later on) as well as the strong similarity between our context and that of ``mostly contracting'' partially hyperbolic diffeomorphisms and to reproduce the line of reasoning of \cite{BV}.

The fact that the supports of two different SRB measures are disjoint can be shown by copying verbatim the proof of \cite[Lemma 2.9]{BV}.

Now let us show that the support of every SRB measure $\mu$ is a minimal set of $\W^{cu}$. Since $\mu$ is a Gibbs $u$-state, its support $\Supp(\mu)$ is a closed and $\W^{cu}$-saturated set. Let $K^{cu}\dans\Supp(\mu)$ be a $\W^{cu}$-minimal set (which is in particular a $G_t$-invariant set). Using Theorem \ref{reconstruirelesetatsdeugibbs} we construct an ergodic Gibbs $u$-state $\mu'$ with $\Supp(\mu')\dans K^{cu}$. Since we proved that the supports of two different ergodic Gibbs $u$-states are disjoint, it must be the case that $\mu=\mu'$ from which we deduce that $K^{cu}=\Supp(\mu)$. The proposition follows.
\end{proof}

\begin{rem}
The argument shows that every $\W^{cu}$-minimal set is the support of a unique SRB measure.
\end{rem}

\subsection{A bijective correspondence between minimal sets and ergodic Gibbs $u$-states}

\begin{proposition}
\label{bijcorminimalgibbsstate}
Let us assume the hypotheses of Proposition \ref{srbugibbs}. Then every minimal set of $\hcF$ supports a unique Gibbs $u$-state and every ergodic Gibbs $u$-state is supported inside a minimal set of $\hcF$.
\end{proposition}

The rest of the section is devoted to the proof of that proposition which, together with Theorem \ref{negativeexponent} and Proposition \ref{srbugibbs}, proves Theorem \ref{SRBmeasuresfgf}.

\paragraph{A geometric property --} For a set $K\dans\hM$ we denote $W^s(K)$ its \emph{stable manifold} i.e. the union of the stable manifolds $W^s(v)$ of elements  $v\in K$. The sets $W^u(K),W^c(K),W^{cs}(K)$ and $W^{cu}(K)$ are defined analogously.

\begin{lemma}
\label{geomproperty}
The following properties hold true for every $v\in\hM$ 
\begin{enumerate}
\item $W^u(W^c(v))=W^{cu}(v)$;
\item $W^s(W^{cu}(v))$ has full volume in $\hL_v$.
\end{enumerate}
\end{lemma}

\begin{proof}
The first property is clear because $W^{cu}(v)$ has been defined as the saturation of $W^u(v)$ in the flow direction and because unstable manifolds are invariant by the flow (i.e. $G_t(W^u(v))=W^u(G_t(v))$). In order to prove the second one, let us work in the universal cover.

Let $L$ be a leaf of $\F$ and $\tL$ be its universal cover. It is compactified by adding the \emph{sphere at infinity} $\tL(\infty)$, defined as the set of equivalence classes of geodesic rays for the relation ``stay at bounded distance'' (see \cite[p.28]{Ba}). Lifts to $T^1\tL$ of stable manifolds of $G_t$ are denoted by $\tW^s(.)$. Manifolds $\tW^{u}(.)$, $\tW^{cs}(.)$ and $\tW^{cu}(.)$ are defined analogously. In order to prove the second property, it is enough to prove the following equality for every $v\in T^1\tL$

\begin{equation}
\label{deccentreinstable}
\tW^s(\tW^{cu}(v))=T^1\tL\moins\tW^{cs}(-v),
\end{equation}
indeed $\tW^{cs}(-v)$ is a strict submanifold of $T^1\tL$ and therefore has volume zero.

For $v\in T^1\tL$ denote by $v(\infty)=\lim_{t\to\infty}c_v(t)\in\tL(\infty)$ and $v(-\infty)=\lim_{t\to\infty}c_v(-t)\in\tL(\infty)$, where $c_v$ is the geodesic directed by $v$. Clearly, $v(-\infty)\neq v(\infty)$ (the geodesic rays generated by $v$ and $-v$ don't stay at bounded distance). Moreover it is well known that $w\in\tW^{cs}(v)$ (resp. $w\in\tW^{cu}(v)$) if and only if $v(\infty)=w(\infty)$ (resp. $v(-\infty)=w(-\infty)$): see \cite[p.72]{Ba}. This implies that $\tW^s(\tW^{cu}(v))\cap\tW^{cs}(-v)=\vide$ for every $v\in T^1\tL$.

Now let $\xi=v(-\infty)$. Let $w\in T^1\tL$ and $\xi'=w(\infty)$. If $w\notin\tW^{cs}(-v)$ then $\xi'\neq\xi$ and there exists a directed geodesic starting at $\xi$ and ending at $\xi'$. This geodesic is precisely the intersection of $\tW^{cu}(v)$ with $\tW^{cs}(w)$. In particular it intersects $\tW^s(w)$. This implies that $w\in\tW^s(\tW^{cu}(v))$.
\end{proof}

\paragraph{Basins of Gibbs $u$-states --}  We now prove that the intersection between the basin of an ergodic Gibbs $u$-state and a typical leaf is large. Recall that the \emph{basin} of $\mu$ is defined as the set
$$\B(\mu)=\left\{v\in\hM;\,\frac{1}{T}\int_0^T\delta_{G_t(v)}dt\To_{T\to\infty}\mu.\right\}.$$

\begin{lemma}
\label{intbasinleaflarge}
Let $\mu$ be an ergodic Gibbs $u$-state for $G_t$. Then there is a Borel set $\X\dans\hM$ full for $\mu$ such that for every $v\in\X$, $\B(\mu)\cap\hL_v$ has full volume in $\hL_v$.
\end{lemma}

\begin{proof}
First notice that $\B(\mu)$ is $\W^s$-saturated. Using the second item of Lemma \ref{geomproperty} as well as the absolute continuity of $\W^s$ inside leaves of $\hcF$ (see for example \cite[Theorem 3.7]{Alvarezu-Gibbs}) it is enough to prove the existence of a Borel set $\X$ full for $\mu$ such that for every $v\in\X$, $\B(\mu)\cap W^{cu}(v)$ has full volume in $W^{cu}(v)$.

Denote by $\X_1\dans\hM$ the set of points $v\in\hM$ such that $\Leb^{cu}$-almost every point of $W^{cu}_1(v)$ belongs to $\B(\mu)$. Since $\mu$ is an ergodic Gibbs $u$-state, $\X_1$ is full for $\mu$.

We define $\X=\bigcap_{n\in\Z} G_n(\X_1)$. Fix $v\in\X$ and denote $v_m=G_m(v)$ for $m\in\Z$.

\paragraph{Claim --} \emph{For every $m\in\Z$ the basin $\B(\mu)$ contains a full volume subset of $W^u(G_{[0,1]}(v_m))$.}\\

Establishing the claim suffices to prove the lemma. Indeed by the first item of \ref{geomproperty} we have $W^{cu}(v)=\bigcup_{m\in\Z}W^u(G_{[0,1]}(v_m))$.

Let us prove this claim. Note that, if we set $\chi=\Min\left|\left|(DG_1)_{|E^u}\right|\right|>1$, we have for $n\geq -m$ 
$$W_{\chi^{n+m}}^u\left(G_{[0,1]}(v_m)\right)\dans G_{n+m}\left(W_1^u\left(G_{[0,1]}(v_{-n})\right)\right),$$
in such a way that
$$W^u\left(G_{[0,1]}(v_m)\right)=\bigcup_{n\in\Z}G_{n+m}\left(W_1^u\left(G_{[0,1]}(v_{-n})\right)\right).$$

Since $v\in\X$ we have $v_{-n}\in\X_1$ for every $n$ and $\Leb^{cu}$-almost every point of $W^u_1(G_{[0,1]}(v_{-n}))$ belongs to $\B(\mu)$. Using that $\B(\mu)$ is $G_t$-invariant this property also holds for $G_{n+m}(W^u_1(G_{[0,1]}(v_{-n})))$. Finally we deduce that $\Leb^{cu}$-almost every $v\in W^u(G_{[0,1]}(v_m))$ belongs to $\B(\mu)$ and the claim, as well as the lemma, follows.
\end{proof}

\paragraph{Pesin manifolds --}  Assume here that $\mu$ is an ergodic Gibbs $u$-state whose transverse Lyapunov exponent is negative. By Proposition \ref{Lyapspectrum}, $\mu$ is \emph{Pesin-hyperbolic}: its Lyapunov exponents are $\neq 0$.

\emph{Pesin's stable manifold theory} (see for example \cite{BV}) implies that $\mu$-almost every $v\in\hM$ admits a \emph{Pesin stable manifold} $W_{Pes}^s(v)$ (of dimension $\dim\,W^s+\codim\,\F$) and that the ``foliation'' $(W^s_{Pes}(v))_{v\in\hM}$ is absolutely continuous (see \cite[p.164]{BV} for the definition).

\emph{Pesin center-stable manifold} $W^{cs}_{Pes}(v)$ is the saturation of $W^s_{Pes}(v)$ in the flow direction. Note that these manifolds are transverse to $\W^u$. We denote respectively by $\W^s_{Pes}$ and $\W^{cs}_{Pes}$ the stable and center-stable Pesin foliations.

\begin{lemma}
\label{basinPesinstable}
Let $\mu$ be an ergodic Gibbs $u$-state for $G_t$ with negative transverse Lyapunov exponent. Then there exists a Borel set $\Y\dans\hM$ full for $\mu$ such that for every $v\in\Y$ there exists a Borel set $\Gamma_v\dans W^u_{loc}(v)$ of full volume such that
\begin{enumerate}
\item $\Gamma_v$ is included in the basin of $\mu$;
\item every point of $\Gamma_v$ admits a Pesin stable manifold.
\end{enumerate}

Moreover there exists $\eps=\eps(v)>0$ as well as a Borel set $\Gamma_{v,\eps}\dans\Gamma_v$ of positive volume such that for every $w\in\hM$ with $\dist(v,w)\leq\eps$, $\W^s_{Pes}(\Gamma_{v,\eps})$ induces a holonomy map between $W^u_1(v)$ and $W^u_1(w)$.

In particular the basin of $\mu$ intersects $W^u_1(w)$ in a set of positive volume whenever $\dist(v,w)\leq\eps$.
\end{lemma}

\begin{proof}
Pesin's theory and Birkhoff's theorem imply the existence of Borel subset $\Y_1\dans\hM$ included in the basin of $\mu$ such that every point of $\Y_1$ possesses a Pesin stable manifold.

Since $\mu$ is a Gibbs $u$-state, inside a foliated chart for $\W^u$ the conditional measure of $\mu$ in local unstable manifolds are equivalent to the Lebesgue measure. This implies in particular that there exists a Borel subset $\Y_2\dans\hM$ full for $\mu$ such that for every $v\in\Y_2$, $\Y_1\cap W^u_1(v)$ has full volume in $W^u_1(v)$.

The remaining part of the lemma follows directly from the absolute continuity of $\W^{cs}_{Pes}$.
\end{proof}

\paragraph{The key lemma --} 
We are now ready to state the main ingredient of the proof of Proposition \ref{bijcorminimalgibbsstate}. In the sequel we adopt the following notation: if $X\dans\hM$ then $\Cl(X)$ denotes its closure.

\begin{lemma}
\label{keylemma}
Assume that every Gibbs $u$-state of $G_t$ has a negative transverse Lyapunov exponent. Let $\mu$ be an ergodic Gibbs $u$-state. Then for every $v\in\Supp(\mu)$, $\mu$ is the only Gibbs $u$-state for $G_t$ supported inside $\Cl(\hL_v)$.
\end{lemma}

\begin{proof}
First note that, by Proposition \ref{srbugibbs}, $\Supp(\mu)$ is a minimal set for $\W^{cu}$ which implies that for every $v\in\Supp(\mu)$ we have $\Supp(\mu)\dans\Cl(\hL_v)$. We deduce that $\Cl(\hL_v)$ does not depend on $v\in\Supp(\mu)$ and in particular we can assume that $v\in\X$, the set constructed in Lemma \ref{intbasinleaflarge}.

For such a $v\in\X$ set $K=\Cl(\hL_v)$ and let $\mu'$ be an ergodic Gibbs $u$-state satisfying $\Supp(\mu')\dans K$ (see Theorem \ref{reconstruirelesetatsdeugibbs}). We will prove that $\mu=\mu'$, and the lemma will follow.

First note that by hypothesis $\mu'$ has a negative transverse Lyapunov exponent so we can apply Lemma \ref{basinPesinstable} to $\mu'$. It provides an element $v'\in K$ such that $\B(\mu')$ contains a full volume subset of $W^u_{loc}(v')$, each point of which admitting a Pesin stable manifold.

Because $v\in\X$, Lemma \ref{intbasinleaflarge} provides $v''\in\hL_v\cap\B(\mu)$ arbitrarily close to $v'$.

Using one more time Lemma \ref{basinPesinstable} we find that $\Gamma=\B(\mu')\cap W^u_1(v'')$ has positive volume in $W^u(v'')$.

The basin of any ergodic $G_t$-invariant measure is $\W^{cs}$-saturated so $W^{cs}(\Gamma)\dans\B(\mu')$. Finally use the absolute continuity of $\W^{cs}$ inside $\hL_{v''}=\hL_v$ to prove that $W^{cs}(\Gamma)$ has positive volume in $\hL_v$. Since by Lemma  \ref{intbasinleaflarge} $\B(\mu)$ has full volume in $\hL_v$ we conclude that $\B(\mu')=\B(\mu)$ and hence that $\mu=\mu'$.
\end{proof}

\paragraph{Proof of Proposition \ref{bijcorminimalgibbsstate} --} We assume in all the following that all Gibbs $u$-states for $G_t$ have negative transverse Lyapunov exponents. We divide the proof of Proposition \ref{bijcorminimalgibbsstate} into two halfs.

\begin{lemma}
\label{firsthalf}
Let $K$ be a minimal set of $\hcF$. There exists a unique Gibbs $u$-state supported inside $K$.
\end{lemma}

\begin{proof}
A minimal set $K$ is closed and $\hcF$-saturated. Therefore applying in $K$ the proof of Theorem \ref{reconstruirelesetatsdeugibbs} provides that Gibbs $u$-states supported in $K$ exist and that ergodic components of such a measure are also supported in $K$. Lemma \ref{keylemma} immediately implies the uniqueness of the Gibbs $u$-state.
\end{proof}

\begin{lemma}
\label{secondhalf}
Every ergodic Gibbs $u$-state is supported inside a minimal set of $\hcF$.
\end{lemma}

\begin{proof}
Let $\mu$ be an ergodic Gibbs $u$-state for $G_t$. We will prove that for every $v\in\Supp(\mu)$, the set $\Cl(\hL_v)$ is minimal for $\hcF$.

Let $v\in\Supp(\mu)$. The set $\Cl(\hL_v)$ is closed and $\hcF$-saturated so it contains $K$, a minimal set for $\hcF$. By Lemma \ref{firsthalf} there is a Gibbs $u$-state $\mu'$ supported inside $K\dans\Cl(\hL_v)$.

But by Lemma \ref{keylemma} $\mu$ is the unique Gibbs $u$-state supported inside $\Cl(\hL_v)$ so it must be the case that $\mu=\mu'$. In particular $\mu$ is supported inside a minimal set of $\F$. 
\end{proof}

\section{Some partially hyperbolic foliated geodesic flows}\label{phfoliatedgeoflows}

In this section, we define the notion of \emph{domination} of projective representations introduced recently in \cite{DeroinTholozan,GKW}, prove the following result, and explain how it implies Theorem \ref{domination}. 

\begin{maintheorem}
\label{equivalencedomination}
Let $(\Pi,M,\Sigma,\F)$ be a foliated $\R\PP^1$-bundle with projective holonomy over a closed surface $\Sigma$ endowed with a hyperbolic metric $m$. Endow $M$ with  an admissible Riemannian metric. Let $\rho:\pi_1(\Sigma)\to\PSL_2(\R)$ denote a Fuchsian representation associated to $m$ and $\hol:\pi_1(\Sigma)\to\PSL_2(\R)$ denote the holonomy representation of $\F$. Then the foliated geodesic flow $G_t$ is partially hyperbolic if and only if $\rho$ dominates $\hol$.
\end{maintheorem}

\subsection{The foliated geodesic flow as a projective cocycle}
\label{suspensioncocycles}
Before proceeding to the proof, let us explain the terms appearing in Theorem \ref{equivalencedomination}.
\paragraph{Fuchsian representation --} We will consider the upper half plane $\Hyp$ endowed with its \emph{Poincar{\'e} metric} $ds^2=(dx^2+dy^2)/y^2$. The group $\PSL_2(\R)$ is identified with the group of its direct isometries.

Let $\Sigma$ be a closed Riemann surface of genus higher than $2$. By uniformization, a smooth hyperbolic metric $m$ on $\Sigma$ gives rise to a \emph{Fuchsian representation}, i.e. a faithful and discrete representation $\rho:\pi_1(\Sigma)\to \PSL_2(\R)$ which is well defined up to conjugacy by an element of $\PSL_2(\R)$. Say two such metrics on $\Sigma$ represent the same point in the \emph{Teichm\"uller space} if one is the image of the other by a diffeomorphism isotopic to the identity. It is equivalent to having their Fuchsian representations conjugated in $\PSL_2(\R)$.

\paragraph{Holonomy representation --} Consider a foliated $\R\PP^1$-bundle with projective holonomy $(\Pi,M,\Sigma,\F)$. It is obtained by the suspension of a representation $\hol:\pi_1(\Sigma)\to\PSL_2(\R)$ (see \cite{CL} for the definition of suspension) called the \emph{holonomy representation} of the foliation. 

Note that holonomy has the following topological interpretation. If $c$ is a closed path on $\Sigma$ and $\gamma\in\pi_1(\Sigma)$ denotes its homotopy class, then there is, for every $x$ belonging to the fiber of $p=c(0)$ a unique lift of $c$ starting at $x$. Then the ending point of this lift depends only on $\gamma$ and is equal to $\hol(\gamma)^{-1}(x)$.

We will endow $M$ with an \emph{admissible metric}, i.e. a smooth Riemannian metric on $M$ whose restriction in each leaf is locally isometric to $m$. Such a metric is conformally equivalent to the usual angle metric in the fibers because the codimension of $\F$ equals $1$.

\paragraph{Foliated geodesic flow as a cocycle --} The differential of $\Pi$ induces an $\R\PP^1$-bundle $\Pi_{\ast}:\hM\to T^1\Sigma$ which is transverse to the foliation $\hcF$. For $w\in T^1\Sigma$ we set $F_{\ast,w}=\Pi_{\ast}^{-1}(w)$.

Since $\Pi$ is a local isometry when restricted to the leaves, it sends geodesics of the leaves on geodesics of the base. As a consequence the foliated geodesic flow $G_t:\hM\to\hM$ projects down via $\Pi_{\ast}$ to the geodesic flow of $T^1\Sigma$ which we denote by $g_t:T^1\Sigma\to T^1\Sigma$. In particular \emph{it preserves the fibers} of $\Pi_{\ast}$ and for every $w\in T^1\Sigma$ and $t\in\R$ the map
$$A_t(w)=(G_t)_{|F_{\ast,w}}:F_{\ast,w}\to F_{\ast,g_t(w)}$$
identifies with the holonomy along the orbit segment $g_{[0,t]}(w)$ and therefore belongs to $\PSL_2(\R)$. Moreover it satisfies the cocycle relation
$$A_{t_1+t_2}(w)=A_{t_1}(g_{t_2}(w))A_{t_2}(w).$$
All this implies that the foliated geodesic flow $G_t$ is a \emph{projective cocycle} over the geodesic flow $g_t$.

\subsection{Domination of representations}
\label{Domrepr}

Recently a notion of domination of representations appeared in the theory of $3$-dimensional Anti-de Sitter geometry \cite{DeroinTholozan,GKW}. 


\paragraph{Domination --} The \emph{translation length} of an element $P\in \PSL_2(\R)$ is by definition
$$l_P=\Inf_{z\in\Hyp}\dist(P\,z,z)\geq 0.$$

\begin{rem}
\label{translength}
If $P$ is an \emph{elliptic} element (i.e. conjugated to a rotation) it has a fixed point in $\Hyp$ and $l_P=0$.

If $P$ is a \emph{parabolic or hyperbolic} element of $\PSL_2(\R)$ (i.e. respectively conjugated to a translation or a homothety) then $l_P$ coincide with the modulus of the logarithm of the derivative at any of its fixed points. In particular it vanishes in the case $P$ is parabolic.
\end{rem}

\begin{lemma}
\label{indepexponent}
Let $P,Q\in\PSL_2(\R)$ such that $Q$ is hyperbolic. Then for every $k\in\Z\moins\{0\}$
$$\frac{l_{P^k}}{l_{Q^k}}=\frac{l_P}{l_Q}.$$
\end{lemma}

\begin{proof}
It is a fairly direct application of Remark \ref{translength} that for every $P\in\PSL_2(\R)$ and $k\in\Z$, $l_{P^k}=kl_P$. The lemma follows.
\end{proof}

The \emph{marked length spectrum} of a projective representation $\phi:\pi_1(\Sigma)\to \PSL_2(\R)$ is by definition the collection $\ell_{\phi}=(l_{\phi(\gamma)})_{\gamma\in\pi_1(\Sigma)}$. Say $\phi_1$ \emph{dominates} $\phi_2$ if there exists $\kappa<1$ such that $\ell_{\phi_2}<\kappa\ell_{\phi_1}$.

\paragraph{Domination by a Fuchsian representation --} We will make use of the following theorem proven independently and with different methods by Gu{\'e}ritaud-Kassel-Wolff in \cite{GKW} (to which we refer for details about the Euler class) and by Deroin-Tholozan in \cite{DeroinTholozan}.

\begin{theorem}
\label{dominationGKWDT}
Let $\Sigma$ be a closed Riemann surface of genus $\g\geq 2$.
\begin{enumerate}
\item Let $\hol:\pi_1(\Sigma)\to \PSL_2(\R)$ be a projective representation with $|\Eu(\hol)|<2\g-2$. Then it is dominated by a Fuchsian representation $\rho:\pi_1(\Sigma)\to \PSL_2(\R)$.
\item Reciprocally any Fuchsian representation $\rho:\pi_1(\Sigma)\to \PSL_2(\R)$ dominates some non-Fuchsian representation whose Euler class can be prescribed in $\{3-2\g,...,0,...,2\g-3\}.$
\end{enumerate}
\end{theorem}

This theorem, together with Theorem \ref{equivalencedomination} gives the fist two items of Theorem \ref{domination}. Proving the last one will be the goal of \S \ref{fuchsianfoliations}. First we need to define partial hyperbolicity.

\subsection{Partially hyperbolic foliated geodesic flows}

\subsubsection{Partially hyperbolic flows}
\label{phflows}

\paragraph{Definition --} A non-singular flow $\Phi_t:N\to N$ on a Riemannian manifold $N$ generated by a vector field $X$ is said to be \emph{partially hyperbolic} if there exists a decomposition of the normal bundle of $X$ of the form
$$\NN_X=E^s_{\NN}\oplus E^c_{\NN}\oplus E^u_{\NN}$$
and two constants $C,\lambda>0$ such that
\begin{enumerate}
\item the bundles $E^s_{\NN},E^c_{\NN},E^u_{\NN}$ are continuous and invariant by the linear Poincar{\'e} flow $\Psi_t:\NN_X\to\NN_X$ of $\Phi_t$;
\item for $x\in N$, $t\geq 0$ and every $v^s\in E^s_{\NN}(x)$ and $v^u\in E^u_{\NN}(x)$,

$$||\Psi_t(x)v^s||\leq C\exp(-\lambda t)||v^s||$$

$$||\Psi_{-t}(x)v^u||\leq C\exp(-\lambda t)||v^u||.$$

\item the decomposition is \emph{dominated} in the sense that for every $t>0$, $x\in N$ and $v^s\in E^s_{\NN}(x)$, $v^c\in E^c_{\NN}(x)$ and $v^u\in E^u_{\NN}(x)$

$$\frac{||\Psi_t(x)v^c||}{||\Psi_t(x)v^u||}\leq C\exp(-\lambda t)\,\,\,\,\,\,\,\,\,\,\,\,\,\,\,\,\,\,\,\,\,\,\,\,\,\,\,\,\textrm{and}\,\,\,\,\,\,\,\,\,\,\,\,\,\,\,\,\,\,\,\,\,\,\,\,\,\,\,\,\frac{||\Psi_{-t}(x)v^c||}{||\Psi_{-t}(x)v^s||}\leq C\exp(-\lambda t).$$
\end{enumerate}

\paragraph{Criterion for domination --} We will give a criterion for partial hyperbolicity due to Ma{\~n}{\'e} \cite{Mane_ergodic} (see also \cite[Proposition 3.4]{KoPo} for a similar statement).

\begin{lemma}
\label{criteriondomination}
Suppose the linear Poincar{\'e} flow $\Psi_t:\NN_X\to\NN_X$ preserves a decomposition $\NN_X=E\oplus F$ such that for every $\Phi_t$-invariant probability measure $\mu$
$$\lambda^+_E(\mu)<\lambda^-_F(\mu),$$
where $\lambda^+_E(\mu)$ and $\lambda^-_F(\mu)$ stand respectively for the greatest Lyapunov exponent of $\mu$ along $E$ and the lowest Lyapunov exponent of $\mu$ along $F$.

Then the decomposition $\NN_X=E\oplus F$ is dominated.
\end{lemma}

\begin{proof}
This argument is quite classical so we only give a glimpse of the proof. Using the invariance and the continuity of the decomposition it is enough to prove that it is dominated for $\Psi_1$. 

\paragraph{Claim --} \emph{For every $x\in N$ there exists a integer $n_x>0$ such that for every $v\in E(x)$ and $w\in F(x)$}
$$\frac{||\Psi_{n_x}(x)v||}{||\Psi_{n_x}(x)w||}<\frac{1}{2}.$$

Proving the previous claim clearly suffices to prove the domination: use the continuity of the Poincar{\'e} flow to prove that $n_x$ is locally constant and the compactness of $N$ to give a uniform upper bound for $n_x$. The domination then follows easily.

Suppose the claim does not hold for some $x\in N$. Then for every integer $n>0$ we have
\begin{equation}
\label{ineq1}
\frac{1}{n}\log|||\Psi_n(x)_{|E(x)}|||-\frac{1}{n}\log m(\Psi_n(x)_{|F(x)})\geq -\frac{\log 2}{n},
\end{equation}
where $|||.|||$ and $m(.)$ stand respectively for the operator \emph{norm and conorm} associated to the norm $||.||$.

Consequently there exists a strictly increasing sequence of integers $(n_k)_{k\geq 0}$ and a $\Phi_1$-invariant measure $\eta$ such that
\begin{equation}
\label{ineq2}
\frac{1}{n_k}\sum_{i=0}^{n_k}\delta_{\Phi_i(x)}\to\eta.
\end{equation}

Setting $h_1(x)=\log|||\Psi_1(x)_{|E(x)}|||$ and $h_2(x)=\log m(\Psi_1(x)_{|F(x)})$ which are continuous functions of $x\in N$ we see that

\begin{eqnarray}
\label{ineq3}
\frac{1}{n_k}\log|||\Psi_{n_k}(x)_{|E(x)}|||&\leq &\frac{1}{n_k}\sum_{i=0}^{n_k-1} h_1\circ\Phi_i(x),\\
\label{ineq4}
\frac{1}{n_k}\log m(\Psi_{n_k}(x)_{|F(x)})&\geq &\frac{1}{n_k}\sum_{i=0}^{n_k-1} h_2\circ\Phi_i(x).
\end{eqnarray}

Putting together \eqref{ineq1}, \eqref{ineq2}, \eqref{ineq3} and \eqref{ineq4} it follows that
$$\lambda^+_E(\eta)-\lambda^-_F(\eta)=\int_Nh_1(x)d\eta(x)-\int_N h_2(x)d\eta(x)\geq 0,$$
where $\lambda^+_E(\eta)$ and $\lambda^-_F(\eta)$ represent respectively the greatest and lowest Lyapunov exponents along $E$ and $F$ of the \emph{diffeomorphism} $\Phi_1$ for $\eta$.

This does not contradict yet our hypothesis. But if $\mu$ denotes the average of the measures $\Phi_{t\ast}\eta$, for $t\in[0,1]$ one easily proves using the commutation formula $\Phi_s\circ\Phi_t=\Phi_t\circ\Phi_s$ that $\mu$ is $\Phi_t$-invariant for every $t$ and has the same Lyapunov exponents as $\eta$. Hence $\mu$ is an invariant measure which does not satisfy the hypothesis of the lemma. This proves the claim by contradiction.
\end{proof}

\subsection{Domination of representations implies partial hyperbolicity}

Here we prove the first half of Theorem \ref{equivalencedomination} by showing that if $\hol$ is dominated by $\rho$ then the corresponding foliated geodesic flow $G_t:\hM\to\hM$ is partially hyperbolic. So let us assume that $\rho,\hol:\pi_1(\Sigma)\to\PSL_2(\R)$ are projective representations such that $\rho$ is Fuchsian and dominates $\hol$.

\paragraph{Domination along periodic orbits --} Let $v\in\hM$ be a periodic point for the foliated geodesic flow and $\mu_v$ be the $G_t$-invariant measure supported by the corresponding periodic orbit.

\begin{lemma}
\label{dominationexponents} If $\rho$ dominates $\hol$ and if $\kappa\in(0,1)$ denotes the domination constant, we have
$$|\lambda^{\tr}(\mu_v)|\leq\kappa.$$
\end{lemma}

\begin{proof}
Let $T_0>0$ be the period of $v$. The projection of the orbit $\OO(v)$ of $v$ is a periodic orbit $\OO(w)$ for $g_t$ where $w=\Pi_{\ast}(v)$.

The free homotopy class of $\OO(w)$ is the conjugacy class of some element $\gamma\in\pi_1(\Sigma)$ and the holonomy map $\tau_w$ over $\OO(w)$ is conjugated to $\hol(\gamma)^{-1}$. In particular it lies in $\PSL_2(\R)$ and, since $\OO(v)$ is closed, has a periodic point in $\R\PP^1$: it has to be conjugated to a rational rotation, a parabolic or hyperbolic element. In the first case the holonomy over $\OO(w)$ is conjugated to an isometry and we clearly have $\lambda^{\tr}(\mu_v)=0$. In the remaining cases, this implies that $v$, which is periodic for $\tau_w$, has to be a fixed point of $\tau_w$.

We deduce two things. Firstly the orbits $\OO(v)$ and $\OO(w)$ have the same length which is equal to $T_0$. Secondly the holonomy map $\h_{G_{[0,T_0]}(v)}$ of $\hcF$ along the closed orbit $G_{[0,T_0]}(v)$ is conjugated to $\hol(\gamma)^{-1}$.

By definition of $\rho$ the length of the closed geodesic $\OO(w)$ is $l_{\rho(\gamma)}=l_{\rho(\gamma)^{-1}}$. By Remark \ref{translength} the logarithm of the derivative at the fixed point $v$ of $\hol(\gamma)^{-1}$ is $\pm l_{\hol(\gamma)^{-1}}$ (note that these quantities are constant on the conjugacy class of $\gamma$).

Using the domination of $\hol$ by $\rho$ one sees that there exists $\kappa\in(0,1)$ independent of $\gamma\in\pi_1(\Sigma)$ such that we have $l_{\hol(\gamma)^{-1}}<\kappa l_{\rho(\gamma)^{-1}}$.

This implies that for every positive integer $k$ we have by Lemma \ref{indepexponent}
\begin{equation}
\label{Lyapperiodic}
\left|\frac{1}{kT_0}\log\,D_v\h_{G_{[0,kT_0]}(v)}\right|=\frac{l_{\hol(\gamma^k)^{-1}}}{l_{\rho(\gamma^k)^{-1}}}=\frac{l_{\hol(\gamma)^{-1}}}{l_{\rho(\gamma)^{-1}}}\leq\kappa.
\end{equation}

Since the left hand side of \eqref{Lyapperiodic} tends to $\lambda^{\tr}(\mu_v)$ as $k$ tends to infinity, the lemma follows.
\end{proof}

\paragraph{Partial hyperbolicity --} Note that the Poincar{\'e} linear flow $\Psi_t$ of  $G_t$ preserves a decomposition
$$\NN_X=E^s_{\NN}\oplus E_{\NN}^c\oplus E^u_{\NN},$$
where $E^c_{\NN}$ denotes the tangent space of the fibers, and $E^s_{\NN},E^u_{\NN}$ represent respectively the orthogonal projections on $\NN_X$ of the stable and unstable directions of the flow. These bundles are $1$-dimensional.

As we saw in \S \ref{transverselyapunovsection} the Lyapunov exponent at $v$ along $E^c_{\NN}$ is precisely $\lambda^{\tr}(v)$.

Moreover, since the metric on the leaves is of constant curvature $-1$, it is clear from the commutation relations between horocyclic and geodesic flows that for every $v\in\hM$ we have
$$\lambda^u(v)=-\lambda^s(v)=1.$$

In order to prove the partial hyperbolicity of the flow $G_t$ we are going to use the criterion stated in Lemma \ref{criteriondomination}. It is enough to prove the

\begin{lemma}
There exists $\kappa\in(0,1)$ such that for every ergodic $G_t$-invariant probability measure $\mu$ we have 
$$|\lambda^{\tr}(\mu)|\leq\kappa.$$
\end{lemma}

\begin{proof}
It is clear that it suffices to treat the case where $\mu$ is an ergodic $G_t$-invariant measure with $\lambda^{\tr}(\mu)\neq 0$. In that case the measure $\mu$ is an ergodic \emph{hyperbolic} measure in the sense of Pesin: all its Lyapunov exponents are non zero.

Since the flow $G_t$ is $C^{\infty}$ we can use \emph{Katok's closing lemma} (see \cite{Katok}). In that case there exists a sequence $(v_k)_{k\geq 0}$ of periodic points for $G_t$ such that
$$\mu_{v_k}\To_{k\to\infty}\mu,$$
in the weak$^{\ast}$-sense, where we recall that $\mu_{v_k}$ is the $G_t$-invariant measure suppported by the periodic orbit $\OO(v_k)$.

The transverse Lyapunov exponent of a $G_t$-invariant measure $\nu$ is in our case given by an integral
$$\lambda^{\tr}(\nu)=\int_{\hM}\log|||\Psi_1(v)_{|E^c_{\NN}(v)}|||d\nu(v),$$
in particular it varies continuously with $\nu$ and we have by Lemma \ref{dominationexponents}
$$|\lambda^{\tr}(\mu)|=\lim_{k\to\infty}|\lambda^{\tr}(\mu_k)|\leq\kappa.$$
\end{proof}

\subsection{Partial hyperbolicity implies domination of representations}

\begin{proof}[End of the proof of Theorem \ref{equivalencedomination}]We assume here that the foliated geodesic flow $G_t$ is partially hyperbolic. We want to find $\kappa<1$ so that for every $\gamma\in\pi_1(\Sigma)$, $l_{\hol(\gamma)}\leq\kappa l_{\rho(\gamma)}$. It is obvious from Remark \ref{translength} that it is enough to treat the case where $\hol(\gamma)$ is hyperbolic because otherwise $l_{\hol(\gamma)}=0$.

Let $\gamma\in\pi_1(\Sigma)$ be such that $\hol(\gamma)$ is hyperbolic. There exists a unique periodic orbit of the geodesic flow of $T^1\Sigma$, denoted by $\OO(w)$, whose free homotopy class is the conjugacy class of $\gamma$. As we have already noticed, the length of the orbit equals $l_{\rho(\gamma)}$.

The holonomy $\tau_w$ over $\OO(w)$ is conjugated to $\hol(\gamma)^{-1}$ and in particular it is a hyperbolic element of $\PSL_2(\R)$ with the same translation length as $\hol(\gamma)^{-1}$. It implies that $\tau_w$ has a repelling fixed point $v$, which is also fixed by all of its powers. By Remark \ref{translength} the  logarithm of the derivative of $\tau_{w}^k$ at $v$ equals $l_{\hol(\gamma^k)^{-1}}$. Now the partial hyperbolicity at $v$ implies the existence of uniform $C,\lambda>0$ such that if $T_0=l_{\rho(\gamma)}=l_{\rho(\gamma)^{-1}}$ denotes the period of $w$ we have for every $k>0$
$$\frac{D\tau^k_w(v)}{e^{kT_0}}\leq Ce^{-k\lambda T_0}.$$

Note that $\log D\tau^k_w(v)=l_{\hol(\gamma^k)^{-1}}=kl_{\hol(\gamma)^{-1}}$ and $kT_0=k\,l_{\rho(\gamma)^{-1}}$. Taking the logarithm, dividing by $k$ and denoting $\kappa=1-\lambda<1$ provides
$$l_{\hol(\gamma)^{-1}}\leq\frac{\log C}{k}+\kappa l_{\rho(\gamma)^{-1}},$$
for every $\gamma\in\pi_1(\Sigma)$ and $k>0$. We deduce that $\hol$ is dominated by $\gamma$. Theorem \ref{equivalencedomination} then follows. \end{proof}

\section{Foliations associated to Fuchsian representations}\label{fuchsianfoliations}

We now turn to the case where the holonomy representation has extremal Euler number. Recall that by \cite{Goldman} this happens if and only if the holonomy representation is Fuchsian. We prove that in that case the foliated geodesic flow is not partially hyperbolic. We go further by computing the transverse Lyapunov exponent of the unique SRB measure.

\begin{maintheorem}
\label{transversefuchsian}
Let $(\Pi,M,\Sigma,\F)$ be a foliated $\R\PP^1$-bundle with projective holonomy over a closed surface $\Sigma$ endowed with a hyerbolic metric $m$. Endow $M$ with  an admissible Riemannian metric. Let $\rho:\pi_1(\Sigma)\to\PSL_2(\R)$ denote a Fuchsian representation associated to $m$ and $\hol:\pi_1(\Sigma)\to\PSL_2(\R)$ denote the holonomy representation of $\F$. Assume that $\hol$ is Fuchsian and let $\chi\geq 1$ be the associated average reparametrization of the geodesic flow. Then the transverse Lyapunov exponent $\lambda^{\tr}$ of the unique SRB measure equals $-\chi$;

In particular when $\rho$ and $\hol$ are not conjugated we find $|\lambda^{\tr}|>1$.
\end{maintheorem}

We start by defining what we call \emph{average reparametrization of the geodesic flow}.

\subsection{Reparametrization of the geodesic flow of a hyperbolic surface}
\label{reparametrizationsection}

\paragraph{Oriented triples --} The geodesic flow of $T^1\Hyp$ shall be denoted by $\widetilde{G}_t$ There is an identification between $T^1\Hyp$ and the set of oriented triples of $\R\PP^1$, denoted by $\SSS$, that associates to every vector $v$ the triple $(pr_+(v),pr_0(v),pr_-(v))$ where
\begin{itemize}
\item $pr_+(v)\in \R\PP^1$ is the extremity of the geodesic ray determined by $-v$;
\item $pr_-(v)\in \R\PP^1$ is the extremity of the geodesic ray determined by $v$;
\item $pr_0(v)\in \R\PP^1$ is the extemity of the geodesic orthogonal to $v$ wich satisfies $pr_+(v)<pr_0(v)<pr_-(v)$ for the orientation.
\end{itemize}

This identification is an equivariance for the actions of $\pi_1(\Sigma)$ given on $T^1\Hyp$ by differentials of hyperbolic isometries, and on $\SSS$ by the diagonal action.

Moreover, the geodesic starting at $pr_+(v)$ and ending at $pr_-(v)$ is parametrized by the point $pr_0(v)$. More precisely, as $v$ evolves according to the geodesic flow, the point $pr_0(v)$ evolves along the differential equation given by the vector field $Y_v$ obtained by pulling back the vector field $x\partial_x$ of $\R\PP^1$ by the unique M\"obius transformation sending respectively $pr_+(v),pr_0(v),pr_-(v)$ on $0,1,\infty$.

\paragraph{Orbit equivalence of the geodesic flows --} There exists a unique homeomorphism $h:\R\PP^1\to\R\PP^1$ which conjugates the actions of $\rho$ and $\hol$ i.e. for every $\gamma\in\pi_1(\Sigma)$
$$h\circ\rho(\gamma)=\hol(\gamma)\circ h.$$
We call $h$ the \emph{boundary correspondence}: it is bih\"older and orientation preserving (see Section 5.9 of Thurston's notes \cite{Thurston} for all these facts). By evaluating $h$ on triples of points, we get an equivariant and bih\"older homeomorphism $H:\SSS\to\SSS$ which descends to the quotient and provides an orbit equivalence between the geodesic flows on the unit tangent bundles corresponding respectively to the metrics $m_1$ and $m_2$. We will conveniently identify $H$ with a homeomorphism of $T^1\Hyp$.

\paragraph{Average reparametrization of the geodesic flow --} Define the function $a:\R\times T^1\Hyp\mapsto\R$ defined for $(t,v)\in\R\times T^1\Hyp$:
$$H\circ \widetilde{G}_t(v)=\widetilde{G}_{a(t,v)}\circ H(v).$$

It can be proven that $a$ descends to a Liouville-integrable (for $m_1$) additive cocycle of $T^1\Sigma$. Birkhoff's additive ergodic theorem then ensures the existence of the following number, that will be referred to as the \emph{average reparametrization of the geodesic flow}, \emph{for Liouville almost every} $v\in T^1\Sigma$
\begin{equation}
\label{averagereparam}
\chi=\lim_{t\to\infty}\frac{a(t,v)}{t}>0.
\end{equation}
%

\begin{theorem}[Thurston, see \cite{Wolpert}]
Let $m_1,m_2$ be two hyperbolic metrics on a closed surface $\Sigma$. Then $\chi\geq 1$ with equality if and only if the two hyperbolic metrics represent the same Teichm\"uller class.
\end{theorem}

\subsection{Non-partially hyperbolic foliated geodesic flows}

\paragraph{Canonical foliated geodesic flow --} Assume for the moment that $\rho=\hol$. By suspension of $\hol$ we obtain the so-called \emph{canonical foliation} $(M^{can},\F^{can})$ that we endowed with an admissible Riemannian metric. We can look at the foliated geodesic flow denoted by $G^{can}_t$ acting on $\hM^{can}$. It can be lifted as a flow of $T^1\Hyp\times\R\PP^1$ still denoted by $\widetilde{G}_t$.  It is possible to consider three sections $\widetilde{\sigma}^{\star,can}:T^1\Hyp\to T^1\Hyp\times\R\PP^1$, $\star=+,0,-$ defined by
$$\widetilde{\sigma}^{\star,can}=(Id,pr_{\star}).$$

One can prove that these sections descend to the quotient and provide three sections $\sigma^{\star,can}:T^1\Sigma\to\hM^{can}$, $\star=+,0,-$. The sections $\sigma^{+,can}$ and $\sigma^{-,can}$ commute with the geodesic flows and are respectively the \emph{sections of largest expansion and contraction} defined in \cite{BGV}. 

\paragraph{Trivialization --} As in \cite{BGV} (see also Section VIII.1.3 of the first author's thesis \cite{AlvarezThese} for the precise construction in this particular context), we can find an equivariant and fiber preserving analytic map $\widetilde{\Phi}:T^1\Hyp\times\R\PP^1\to T^1\Hyp\times\R\PP^1$ which:
\begin{itemize}
\item sends respectively the sections $\widetilde{\sigma}^{+,can},\widetilde{\sigma}^{0,can},\widetilde{\sigma}^{-,can}$ on the sections corresponding to the constant functions respectively equal to  $0,1$ and $\infty$.
\item sends the vector field $Y_v$ defined above on the vector field $x\partial_x$.
\end{itemize}
For this consider $\widetilde{\Phi}(v,x)=(v,P_v(x))$ for $(v,x)\in T^1\Hyp\times\R\PP^1$, where $P_v$ denotes the unique M{\"o}bius transform sending  the triple $(pr_{+}(v),pr_{0}(v),pr_{-}(v))$ on $(0,1,\infty)$

\paragraph{Horizontal and vertical components --} Define on $T^1\Hyp\times\R\PP^1$ the vector field $\widetilde{X}$ which generates the lift to $T^1\Hyp\times\R\PP^1$ of the canonical geodesic flow, which is still denoted by $\widetilde{G}_t$.

The \emph{vertical component} of $\widetilde{X}$ is by definition the vector field $\widetilde{Y}$ on $T^1\Hyp\times\R\PP^1$ which is tangent to the fibers $\{v\}\times\R\PP^1$ and induces on any such fiber the vector field $Y_v=P_v^{\ast}(x\partial_x)$ where $P_v$ has been defined above. The following lemma is essentially due to Bonatti, G{\'o}mez-Mont and Vila (see \cite[\S 8.2]{BGV}), we give below a glimpse of its proof in our context.

\begin{lemma}
\label{decomposition}
\begin{enumerate}
\item The vector field $\widetilde{Y}$, as well as the sum $\widetilde{Z}=\widetilde{X}+\widetilde{Y}$, are invariant by the diagonal action of $\pi_1(\Sigma)$ on $T^1\Hyp\times\R\PP^1$ and commute with $\widetilde{X}$.
\item The vector field $\widetilde{\Phi}_{\ast}\widetilde{Y}$ is tangent to the fibers $\{v\}\times\R\PP^1$ and induces the vector field $x\partial_x$ in these fibers;
\item The flow of $\widetilde{\Phi}_{\ast}\widetilde{Z}$ preserves the fibers and sends fiber to fiber as the identity.
\end{enumerate}
\end{lemma}

\begin{proof}
In order to prove the first item it is enough to see that $\widetilde{Y}$ commutes with the foliated geodesic flow, which is a direct consequence of the definition of the three sections: we only have to check that $Y_v=Y_{\widetilde{G}_t(v)}$ for every $v\in T^1\Hyp$ (this is left to the reader).

The second item is a trivial consequence of the definition of $\widetilde{Y}$.

In order to prove the third item it is enough to prove that the flow $\widetilde{Z}_t$ of $\widetilde{Z}$ is fiber preserving and commutes with the three sections $\widetilde{\sigma}^{\star,can}$, $\star=+,0,-$. Indeed under these conditions the flow of $\widetilde{\Phi}_{\ast}\widetilde{Z}$ preserves the fibers (since $\widetilde{\Phi}$ is fiber preserving) and sends fiber to fiber as a M{\"o}bius transformation fixing $0,1$ and $\infty$. As a consequence it has to be transversally the identity.

Since $\widetilde{X}$ and $\widetilde{Y}$ commute, and since their flows both send fibers on fibers we get that $\widetilde{Z}_t$ does as well.

Since $\widetilde{\sigma}^{\pm,can}$ are zeros of $Y_v$ for every $v$ and since $\widetilde{G}_t$ commutes with these sections, one easily gets that $\widetilde{Z}_t$ commutes as well with these sections.

By definition $\widetilde{Y}_t(\widetilde{\sigma}^{0,can}(v))=(v,pr_0(\widetilde{G}_t(v)))$ and $\widetilde{G}_t(\widetilde{\sigma}^{0,can}(v))=(\widetilde{G}_t(v),pr_0(v))$. Since $\widetilde{Y}$ and $\widetilde{X}$ commute we have
$$\widetilde{Z}_t(\widetilde{\sigma}^{0,can}(v))=\widetilde{Y}_t\circ \widetilde{G}_t(\widetilde{\sigma}^{0,can}(v))=(\widetilde{G}_t(v),pr_0(\widetilde{G}_t(v)))=\widetilde{\sigma}^{0,can}(\widetilde{G}_t(v)),$$
and the lemma follows.
\end{proof}

All the objects above are equivariant for the diagonal action of $\pi_1(\Sigma)$: they all descend to the quotient (the notation of these objects is doing just by omission of the ``tilde''-character). The foliated geodesic flow of $\hM^{can}$ then satisfies for every time $t\in\R$ $G^{can}_t=Y_{-t}\circ Z_t$.

\paragraph{The SRB measure --} The unique SRB measure for $G^{can}_t$ is precisely $\mu^{+,can}=\sigma^{+,can}\,_{\ast}\Liouv$ (see \cite[Theorems 1.1 and 1.6]{BGV}). With the description made above, we can prove the following
\begin{proposition}
\label{canonical}
The transverse Lyapunov exponent of the canonical foliated geodesic flow for its unique SRB measure is equal to $-1$.
\end{proposition}

\begin{proof}
First the value of the transverse Lyapunov of the SRB measure is independent of the choice of a transverse metric. We are going to use the pullback by $\Phi$ of the usual metric of the fibers (identified with $\R\PP^1$) in order to compute it.

It follows from the discussion above that we can write $G^{can}_t=Y_{-t}\circ Z_t$ where, after the smooth fiber-preserving change of coordinates $\Phi$, $Y_t$ coincide in each fiber with the flow $(x,t)\mapsto e^tx$ and $Z_t$ induces the identity map between fibers. The metric of the fibers is by definition sent onto the usual metric of $\R\PP^1$.

A typical point for the SRB measure is given by $\sigma^{+,can}(v)$. To compute the transverse Lyapunov at this point, it is enough to compute the Lyapunov exponent at $0$ of $(x,t)\mapsto e^{-t}x$, which is $-1$.
\end{proof}

\paragraph{General Fuchsian foliation --} We now turn to the case of two a priori different Fuchsian representations $\rho,\hol$. Consider $H:T^1\Hyp\to T^1\Hyp$, the reparametrization of the geodesic flow. It gives an equivariant bih\"older orbit equivalence $(H,Id):T^1\Hyp\times\R\PP^1\to T^1\Hyp\times\R\PP^1$ between foliated geodesic flow corresponding to $\hol$ and the canonical one.

This in turn provides a bih{\"o}lder orbit equivalence $\widehat{H}:\hM\to\hM$ such that for every $v\in\hM$ and $t\in\R$:
$$G_t(v)=\widehat{H}^{-1}\circ G^{can}_{a(t,v)}\circ\widehat{H}(v)=\widehat{H}^{-1}\circ Y_{-a(t,v)}\circ Z_{a(t,v)}\circ\widehat{H}(v).$$

Using now that $\widehat{H}$, although being only H{\"o}lder continuous in the horizontal direction, is smooth in the fiber direction, and the easy fact that the SRB measure of $G_t$ is precisely given by $\mu^+=\widehat{H}_{\ast}\mu^{+,can}$, we conclude the proof of Theorem \ref{transversefuchsian}:

\begin{center}
\emph{The transverse Lyapunov exponent of $G_t$ for its unique SRB measure is equal to $-\chi$.}
\end{center}

\section{Appendix. Negatively curved metrics in the leaves of a foliation by surfaces}
\label{neg_curved}

In this paragraph $M$ is a closed manifold endowed with a smooth foliation by surfaces $\F$ and with a smooth Riemannian metric $g$. Our goal is to prove Ghys' theorem \ref{negativeurvature}: we want to find in the conformal class of $g$ a metric whose restriction to each leaf has negative Gaussian curvature.

\subsection{Harmonic measures and Gauss-Bonnet theorem}

\paragraph{Harmonic measures --} There is well defined \emph{foliated Laplace-Beltrami operator} defined on the space $C^{0,2}(M)$ of continuous functions $\phi:M\to\R$ which are of class $C^2$ inside the leaves that we denote by $\DeltaF$. By definition for every $x\in M$ and $\phi\in C^{0,2}(M)$ we have $\DeltaF\phi(x)=\Delta_{L_x}\phi(x)$ where $L_x$ is the leaf of $x$ and $\Delta_{L_x}$ denotes the Laplace operator for the restricted metric $g_{L_x}$.

\begin{defi}[Harmonic measures]
A harmonic measure for $\F$ is a probability measure $m$ on $M$ such that for every $\phi\in C^{0,2}(M)$
\begin{equation}
\label{zeroint}
\int_M\DeltaF\phi\,\,dm=0.
\end{equation}
\end{defi}

Garnett proved in \cite{Garnett} the existence of harmonic measures.

\begin{rem}
\label{remouille}
Note that since $M$ is compact, by using a partition of unity and convolution, we can find for every $\phi\in C^{0,2}(M)$ a sequence $(\phi_n)_{n\in\N}$ of smooth functions on $M$ converging uniformly to $\phi$ and whose derivatives of first and second orders inside the leaves converge uniformly to those of $\phi$. This means in particular that $\DeltaF\phi_n\to\DeltaF\phi$ uniformly. Hence to prove that $m$ is harmonic it is enough that it vanishes on the Laplacians of smooth functions, i.e. that \eqref{zeroint} holds for every $\phi\in C^{\infty}(M)$.
\end{rem}

The reader will notice that this remark, together with Candel's simultaneous uniformization theorem \cite{Candel_unif} proves Theorem \ref{negativeurvature}. We will carry on the presentation of Ghys' argument which we believe has the merit of being simple and independent of that theorem.

\paragraph{Foliations by hyperbolic surfaces --} Using \emph{isothermal coordinates} the metric $g$ gives $\F$ a structure of Riemann surface foliation (see \cite[Theorem 3.2.]{Candel_unif} for more details).

Say $\F$ is a \emph{foliation by hyperbolic surfaces} if the universal cover of every leaf is conformally equivalent to the unit disc $\D$. As noted in \cite{Candel_unif} the property of being a foliation by hyperbolic surfaces is topological and is independent of the choice of a metric $g$. It comes from \cite[Lemma 2.1]{Ghys_unif} that

\begin{proposition}
\label{hyptype}
Let $(M,\F)$ is a closed manifold foliated by surfaces and $g$ be a Riemannian metric on $M$. If $M$ does not possess a transverse invariant measure, then all of its leaves are hyperbolic.
\end{proposition}

\paragraph{Gauss-Bonnet theorem --} Hereafter we let $\kappa(x)$ denote the Gaussian curvature at $x$ of the leaf $L_x$. This is a continuous function of $x\in M$.

Even if a priori we don't have $\kappa<0$ everywhere, Ghys proved in \cite{Ghys_GB} the following foliated analogue of Gauss-Bonnet theorem.

\begin{theorem}[Ghys]
\label{gaussbonnet}
Let $(M,\F)$ be a foliated manifold foliated by hyperbolic surfaces and $g$ be a Riemannian metric on $M$. Then for every harmonic measure $m$ we have
$$\int_M \kappa\,dm<0.$$
\end{theorem}

\subsection{Proof of Theorem \ref{negativeurvature}} 

\paragraph{Conformal change of metric --} Let $(M,\F)$ be a closed manifold foliated by surfaces and $g$ be a Riemannian metric. Let $\phi:M\to\R$ be a smooth function and $g'=e^{2\phi}g$. The Gaussian curvature of $g'$ at $x$ of $L_x$, denoted by $\kappa'(x)$, is related to $\kappa(x)$ by the following formula (see \cite{Ghys_GB})
\begin{equation}
\label{changecurvature}
\kappa'(x)=e^{-2\phi(x)}\left(\kappa(x)-\DeltaF\phi(x)\right).
\end{equation}

Recall that we want to prove Theorem \ref{negativeurvature} which states the existence of a metric $g'$ conformally equivalent to $g$ such that $\kappa'<0$ everywhere. Theorem \ref{negativeurvature} is then a consequence of the following key lemma (see \cite{Ghys_GB,Ghys_unif}  and \cite[Lemma 3.5]{DeroinKleptsyn}).

\begin{lemma}[Ghys]
\label{keylemmaGhys}
Assume that all leaves of $\F$ are hyperbolic. Then there exists a smooth function $\phi:M\to\R$ such that for every $x\in M$
$$\kappa(x)-\DeltaF\phi(x)<0.$$
\end{lemma}

\begin{proof}
Let $\CC=C^0(M)$ denote the Banach space of continuous functions of $M$ endowed with the supremum norm $||.||_{\infty}$. Let $\HH$ be the closure in $\CC$ of the space $\{\,\DeltaF\phi;\phi\in C^{\infty}(M)\}$. This is a closed subspace of a Banach space so the quotient $\CC/\HH$ is naturally a Banach space and $\Pi:\CC\to\CC/\HH$ is continuous and open.

\paragraph{Claim --}\emph{The space $\HH ar$ of harmonic measures is identified with the space of positive and continuous linear forms of $\CC/\HH$.}

\begin{proof}[Proof of the claim]Define the \emph{orthogonal complement} of $\HH$ as the closed space of continuous linear forms $m$ defined in $\CC$ such that $m(h)=0$ for every $h\in\HH$. This space identifies isometrically with the topological dual of $\CC/\HH$.

A harmonic measure is a Radon measure vanishing on every element $\DeltaF\phi$, $\phi\in C^{\infty}(M)$ so it must vanish on every element of $\HH$, which is by definition a uniform limit of such functions. Now by Riesz representation theorem, positive elements of the orthogonal complement of $\HH$ are Radon measures vanishing in particular on every laplacian $\DeltaF\phi$, $\phi\in C^{\infty}(M)$: they are harmonic measures by Remark \ref{remouille}.
\end{proof}

Consider now the open cone $\Lambda^-\dans\CC$ of negative continuous functions and its projection $\widehat{\Lambda}^-=\Pi(\Lambda^-)\dans\CC/\HH$. Let $\hat{\kappa}=\Pi(\kappa)\in\CC/\HH$.

\paragraph{Claim --} \emph{We have $\hat{\kappa}\in\widehat{\Lambda}^-$.}

\begin{proof}[Proof of the claim]
Suppose the contrary. By continuity and openness of $\Pi$, $\widehat{\Lambda}^-$ is a nonempty open convex subset of the normed vector space $\CC/\HH$ and $\hat{\kappa}\notin\widehat{\Lambda}^-$. Hahn-Banach's theorem (see \cite[Lemme I.3]{Brezis}) states that there exists $m\in(\CC/\HH)'$ and $a\in\R$ such that for every $u\in\widehat{\Lambda}^-$, $m(u)< a=m(\hat{\kappa})$. 

Let us evaluate $m$ on elements of the form $\lambda u$, $u\in\widehat{\Lambda}^-$, $\lambda>0$. Letting $\lambda$ tend to infinity we see that $m\leq 0$ on $\widehat{\Lambda}^-$. Letting $\lambda$ tend to zero we see that $a\geq 0$.

Hence the linear form $m$, correspond to an element of the orthogonal complement of $\HH$ which is nonpositive on nonpositive functions.

Using the first claim we see that we found a harmonic measure $m$ such that $\int_\M\kappa\,dm= a\geq 0$: this contradicts Ghys' Gauss-Bonnet Theorem \ref{gaussbonnet}.
\end{proof}

Finally to conclude the proof of Lemma \ref{keylemmaGhys} note that the previous claim implies that there exists $h\in\HH$ such that $\kappa-h\in\Lambda^-$ i.e. $\kappa-h<0$ on $M$. Now by definition of $\HH$ there must exist a smooth function $\phi$ such that $\kappa-\DeltaF\phi<0$ on $M$.
\end{proof}

\section*{Acknowledgments} 
S.A was supported by a post-doctoral grant at IMPA financed by CAPES. During the preparation of this work, the authors benefited from the excellent working conditions provided by IMPA and UFF. We wish to thank both institutions.

We were strongly influenced by the work of Christian Bonatti, Xavier G{\'o}mez-Mont and Matilde Mart{\'i}nez, and it is a pleasure to thank them here. Special thanks go to {\'E}tienne Ghys who explained us his very nice proof of Theorem \ref{negativeurvature} and kindly authorized us to reproduce his argument. The question of knowing when the foliated geodesic flow is partially hyperbolic emerged after a discussion with Rafael Potrie. We are grateful to him. Finally we thank the referees for their valuable comments.

\bibliographystyle{plain}

\begin{flushleft}
{\scshape S\'ebastien Alvarez}\\
Centro de Matem\'atica (CMAT), Universidad de la Rep\'ublica\\
Igua 4225, Montevideo, 11400, Uruguay\\
email: salvarez@cmat.edu.uy

\smallskip

{\scshape Jiagang Yang}\\
Departamento de Geometria, Instituto de Matem{\'a}tica e Estat{\'i}stica, Universidade Federal Fluminense (UFF)\\
Rua M{\'a}rio Santos Braga S/N, Niteroi, 24020-140, Brasil.\\
email: yangjg@impa.br
\end{flushleft}

\end{document}